\documentclass[12pt]{amsart}
\usepackage[T1]{fontenc}
\usepackage[utf8]{inputenc}
\usepackage{amsmath,amssymb,amsthm,mathrsfs}
\usepackage{graphicx}
\usepackage{enumerate} 
\usepackage{lmodern}
\usepackage{mathrsfs} 
\usepackage{dsfont}
\usepackage{color}
\usepackage[all]{xy}
\usepackage{hyperref}
\usepackage[margin=3.5cm]{geometry}
\usepackage{caption}
\usepackage{subcaption}
\usepackage{stmaryrd}
\usepackage{tikz}
\usepackage{pgfplots}
\pgfplotsset{compat=1.15}
\usepackage{mathrsfs}
\usetikzlibrary{arrows}

\newtheorem{theorem}{Theorem}[section]
\newtheorem{lemma}[theorem]{Lemma}
\newtheorem{proposition}[theorem]{Proposition}
\newtheorem{corollary}[theorem]{Corollary}
\theoremstyle{definition}
\newtheorem{definition}[theorem]{Definition}

\usepackage{tikz}
\usepackage{tikz-cd}
\theoremstyle{remark}
\newtheorem{remark}[theorem]{Remark}

\numberwithin{equation}{section}

\DeclareMathOperator{\Mod}{Mod}

\newcommand{\V}{\mathbb{V}}

\newcommand{\RR}{\mathrm{R}}

\newcommand{\R}{\mathbb{R}}

\newcommand{\Z}{\mathbb{Z}}

		\newcommand{\RHom}[1][]{\RR\mathrm{Hom}_{\raise1.5ex\hbox to.1em{}#1}}
\newcommand{\Hom}[1][]{\mathrm{Hom}_{\raise1.5ex\hbox to.1em{}#1}}

\newcommand{\rhom}[1][]{{\RR\mathscr{H}\mspace{-3mu}om}_{\raise1.5ex\hbox to.1em{}#1}}

\newcommand{\opnorm}[1]{{\left\vert\kern-0.25ex\left\vert\kern-0.25ex\left\vert #1 
	\right\vert\kern-0.25ex\right\vert\kern-0.25ex\right\vert}}

\newcommand{\D}{\textnormal{D}}
\newcommand{\kk}{\textbf{k}}

\newcommand{\Rr}{\textnormal{R}}

\newcommand{\CF}{\textnormal{CF}}
\newcommand{\PL}{\textnormal{PL}}
\newcommand{\Ho}{\textnormal{H}}
\newcommand{\Gr}{\textnormal{Gr}}
\newcommand{\dk}{~\mathrm{d} \chi}

\begin{document}	

\title{Persistence and the Sheaf-Function Correspondence}

\address{}
\curraddr{kmjhmkljh}
\email{}
\thanks{}

\author{Nicolas Berkouk}
\date{\today}

\begin{abstract}
The sheaf-function correspondence identifies the group of constructible functions on a real analytic manifold $M$ with the Grothendieck group of constructible sheaves on $M$. When $M$ is a finite dimensional real vector space, Kashiwara-Schapira have recently introduced the convolution distance between sheaves of $\kk$-vector spaces on $M$. In this paper, we characterize distances on the group of constructible functions on a real finite dimensional vector space that can be controlled by the convolution distance through the sheaf-function correspondence. Our main result asserts that such distances are almost trivial: they vanish as soon as two constructible functions have the same Euler integral. We formulate consequences of our result for Topological Data Analysis: there cannot exists non-trivial additive invariants of persistence modules that are continuous for the interleaving distance.
\end{abstract}

\maketitle

\section{Introduction}

Inspired by persistence theory from Topological Data Analysis (TDA)  \cite{oudot2015persistence,dey2022computational}, Kashiwara and Schapira have recently introduced the convolution distance between (derived) sheaves on a finite  dimensional real normed vector space \cite{KS18}.  This construction has found important applications, both in TDA --where it allows us to express stability of certain constructions with respect to noise in datasets-- \cite{BG18, BGO19, BP21, BP22} and in symplectic topology \cite{AsIke20,AsIke22, GC22}.  A challenging research direction,  of interest to these two fields,  is to associate numerical invariants to a sheaf on a vector space, which satisfy a certain form of stability with respect to the convolution distance. 

To do so, the TDA community has been mostly using module-theoretic notions,  such as the rank-invariant \cite{CZ07, Cer13},  the Hilbert function or the graded Betti numbers \cite{HOST17, Ber21, OudSco21, LW22}.  From a sheaf-theoretic perspective,  	a natural numerical invariant to consider is the local Euler characteristic,  which is a constructible function that encodes exactly the class of a sheaf in the Grothendieck group,  by a result of Kashiwara \cite{Kash85}.  This is usually called the sheaf-function correspondence.

The group of constructible functions is well-understood  and has the surprisingly nice property that the formalism of Grothendieck's six operations descend to it through the sheaf-function correspondence \cite{Sch91}.  In particular, this allows one to introduce well-behaved transforms of constructible functions, such as the Radon or hybrid transforms \cite{Sch95,BG09,Leb21,KiSa21}.  Constructible functions have already been successfully applied in several domains, such as target enumeration for sensor networks, image and shape analysis \cite{BG09, CMT18},  though the question of their stability with respect to noise in the input data remain poorly understood \cite[Chapter 16]{CGR12}. For instance, in the context of predicting clinical outcomes in glioblastoma \cite{Crawford_2019}, the authors overcome numerical instability by introducing an \emph{ad-hoc} smoothed version of the Euler Characteristic Transform (ECT) \cite{CMT18}, that is empirically more stable than the standard ECT, though no theoretical stability result is provided. 

 In this context,  a natural question is to understand the stability of the sheaf-function correspondence.  The convolution distance  is already considered as a meaningful measurement of dissimilarity between sheaves,  both in applied and pure contexts. Therefore,  we propose in this work to characterize the pseudo-extended metrics on the group of constructible functions on a vector space,  which are controlled in an appropriate sense by the convolution distance through the sheaf-function correspondence.  Our main result (Theorem \ref{th:main}) asserts that these metrics are almost trivial: they vanish as soon as two constructible functions have the same Euler integral.

 Thanks to results by the author and F. Petit \cite{BP21}, we are able to transfer Theorem \ref{th:main} in the context of persistence modules. In particular, we obtain that every additive invariants of compactly generated constructible persistence modules that is continuous for the interleaving distance is trivial (Theorems \ref{th:derivedpersistence} and \ref{th:mainpers}).
 
 We acknowledge that similar results have been obtained independently by Biran, Cornea and Zhang in \cite{BCZ22}, in the specific case of constructible functions over a one-dimensional vector space, with the aim to study $K$-theoretical invariants of triangulated persistence categories.  
 

\section{Sheaves and Constructible Functions}

In this section, we introduce the necessary background and terminology on constructible sheaves and constructible functions.

\subsection{Sheaf-Function correspondence}

Throughout this paper, $\kk$  denotes a  field. For a topological space $X$, we denote by $\textnormal{Mod}(\kk_X)$ the category of sheaves of $\kk$-vector spaces on $X$, and $\D^b(\kk_X)$ its bounded derived category. Let $M$ be a real analytic manifold. The definitions and results of this section are exposed in detail in \cite[Chapters 8 \& 9.7]{KS90}.
\begin{definition}
 A sheaf $F \in \textnormal{Mod}(\kk_M)$ is $\R$-constructible (or constructible for simplicity), if there exists a locally finite covering of $M$ by subanalytic subsets $M = \cup_\lambda M_\lambda$ such that for all $M_\lambda$ and all $j\in \Z$, the restriction $F_{|M_\lambda}$ is locally constant and of finite rank.
\end{definition}

We denote by $\textnormal{Mod}_{\R c}(\kk_M)$ the full subcategory of $\textnormal{Mod}(\kk_M)$ consisting of constructible sheaves and by $\D^b_{\R c}(\kk_M)$ the full subcategory of $\D^b(\kk_M)$ whose objects are sheaves $F \in \D^b(\kk_M)$ such that $\Ho^j(F) \in \textnormal{Mod}_{\R c}(\kk_M)$ for $j\in \Z$. It is well-known \cite[Th. 8.4.5]{KS90} that the functor $ \D^b(\textnormal{Mod}_{\R c}(\kk_M))\longrightarrow \D^b_{\R c}(\kk_M) $ is an equivalence. The objects of $\D^b_{\R c}(\kk_M)$ are still called constructible sheaves.

\begin{definition}
A constructible function on $M$ is a map $\varphi : M \longrightarrow \Z$ such that the fibers $\varphi^{-1}(m)$ are subanalytic subsets, and the family $\{\varphi^{-1}(m)\}_{m \in \Z}$ is locally finite in $M$.
\end{definition}

We denote by $\CF(M)$ the group of constructible functions on $M$. All the remaining results of the section are contained in \cite[Chapter 9.7]{KS90}.

\begin{theorem}
Let $\varphi \in \CF(M)$, there exists a locally finite family of compact contractible subanalytic subsets $\{X_\lambda\}$ such that $\varphi = \sum_\lambda C_\lambda \cdot 1_{X_\lambda}$, with $C_\lambda \in \Z$.

\end{theorem}

\begin{proposition}
Let $\varphi \in \CF(M)$ with compact support. For any finite sum decomposition $\varphi =\sum_\lambda C_\lambda \cdot 1_{X_\lambda} $, where the $X_\lambda$'s are subanalytic compact and contractible, the quantity $\sum_\lambda C_\lambda$ only depends on $\varphi$.
\end{proposition}

\begin{definition}
With the above notations, one defines $\int \varphi \dk := \sum_\lambda C_\lambda$.
\end{definition}

To any constructible sheaf $F \in \D^b_{\R c }(\kk_M)$, it is possible to associate a constructible function $\chi(F) \in \CF(M)$, called the local Euler characteristic of $F$, and defined by:

\[\chi(F) (x) = \chi(F_x) = \sum_{i\in \Z} (-1)^i \textnormal{dim}_\kk(\Ho^i(F)_x).\]

\noindent It is clear that for any distinguished triangle $F' \longrightarrow F \longrightarrow F'' \stackrel{+1}{\longrightarrow}$ in $\D^b_{\R c}(\kk_M)$, one has $\chi(F) = \chi(F') + \chi(F'')$. Therefore, $\chi$ factorizes through the Grothendieck group $K(\D^b_{\R c}(\kk_M))$ and there is a well-defined morphism of groups $K(\D^b_{\R c}(\kk_M)) \longrightarrow\CF(M)$ mapping $[F]$ to $\chi(F)$.

\begin{theorem}[Sheaf-function correspondence]\label{th:correspondence}
The morphism $K(\D^b_{\R c}(\kk_M)) \longrightarrow\CF(M)$ is an isomorphism of groups.
\end{theorem}

\begin{remark}
The proof of the above theorem in \cite[Theorem 9.7.1]{KS90} do not make use of the characteristic $0$ hypothesis stated at the beginning of  \cite[Chapter 9]{KS90} for expository convenience, and therefore extends to any field.
\end{remark}

\begin{lemma}\label{lem:integralsection}
Let $F \in \D^b_{\R c }(\kk_M)$ with compact support, then:

\[\int \chi(F) \dk = \chi\left(\Rr \Gamma(M; F)\right) = \sum_{i\in \Z}(-1)^i \textnormal{dim}_\kk \left(\Ho^i(M;F)\right). \]
\end{lemma}

We review briefly the construction of the direct image operation for constructible functions. Let $f : X \longrightarrow Y$ be a morphism of real analytic manifolds and $\varphi \in \CF(X)$ such that $f$ is proper on $\textnormal{supp}(\varphi)$. Then, for each $y\in Y$, $\varphi \cdot 1_{f^{-1}(y)}$ is constructible and has compact support. 

\begin{definition}
Keeping the above notations, one defines the function $f_\ast \varphi : Y \longrightarrow \Z$ by:

\[(f_\ast \varphi)(y) := \int \varphi \cdot 1_{f^{-1}(y)} \dk.\]
\end{definition}

\begin{remark}
With $a_X : X \longrightarrow \{pt\}$, one has $a_{X\ast} \varphi = (\int \varphi \dk ) \cdot 1_{\{pt\}}  $.
\end{remark}

\begin{theorem} \label{th:functoriality}
\begin{enumerate}
    \item Let $\varphi \in \CF(X)$ and $f : X \longrightarrow Y$ be a morphism of real analytic manifolds such that $f$ is proper on $\textnormal{supp}(\varphi)$. Then $f_\ast \varphi$ is constructible on $Y$.
    
    \item Let $F \in \D^b_{\R c}(\kk_X)$ such that $\chi(F) = \varphi$. Then $\chi(\Rr f_\ast F) = f_\ast \chi (F) = f_\ast \varphi.$
    
    \item Let $g : Y \longrightarrow Z$ be another morphism of real analytic manifold, such that $g \circ f$ is proper on $\textnormal{supp} (g \circ f)$. Then: 
    
    \[(g \circ f)_\ast \varphi = g_\ast f_\ast \varphi.\]
\end{enumerate}
\end{theorem}

\subsection{Convolution distance}

We consider a finite dimensional real vector space $\V$ endowed with a norm $\|\cdot\|$. We equip $\V$ with the usual topology. Following \cite{KS18}, we briefly present the convolution distance, which is inspired from the interleaving distance between persistence modules \cite{Chazal2009}. We introduce the following notations: 

\begin{equation*}
	s : \V \times \V \longrightarrow \V, ~~~s(x,y) = x + y
\end{equation*}
\begin{equation*}
	p_i : \V \times \V \longrightarrow \V ~~(i=1,2) ~~~p_1(x,y) = x,~p_2(x,y) = y.
\end{equation*} 

The convolution bifunctor $- \star - \colon \D^b(\kk_\V)\times \D^b(\kk_\V) \longrightarrow \D^b(\kk_\V)$ is defined as follows. For $F, \; G \in \D^b(\kk_\V)$, we set
	\begin{align*}
		F \star G := \Rr s_! (F \boxtimes G).
	\end{align*}

	\noindent For $r \geq 0$ and $x\in \V$, let $B(x,r) = \{v \in \V \mid \|x-v\| \leq r\}$ and $K_r := \kk_{B(0,r)}$. For $r < 0 $, we set $K_r := \kk_{\{x \in \V \mid \, \|x\| < -r\}}[n]$ (where $n$ is the dimension of $\V$). 
	
\noindent	The following proposition is proved in \cite{KS18}.
	
	\begin{proposition}
		
		\label{P:propertiesofconvolution} Let $\varepsilon, \varepsilon'\in \R$ and $F \in \D^b(\kk_\V)$. There are functorial isomorphisms 
		
		\begin{equation*}
			( K_{\varepsilon}  \star K_{\varepsilon'}) \star F \simeq K_{\varepsilon + \varepsilon'} \star F ~~~ and ~~~ K_0 \star F\simeq F.
		\end{equation*}
	\end{proposition}

	If $\varepsilon \geq \varepsilon' \geq 0 $, there is a canonical morphism 
	$\chi_{\varepsilon, \varepsilon'} \colon K_{\varepsilon}\longrightarrow K_{\varepsilon'}$ in $\D^b(\kk_\V)$. It induces a canonical morphism $\chi_{\varepsilon, \varepsilon'} \star F \colon K_{\varepsilon} \star F \longrightarrow  K_{\varepsilon'} \star F $. 
	In particular when $\varepsilon' = 0$, we get
	\begin{equation}
		\chi_{\varepsilon,0} \star F \colon K_{\varepsilon} \star F \longrightarrow  F.
	\end{equation}
	
	Following \cite{KS18}, we recall the notion of $\varepsilon$-isomorphic sheaves. 
	
	\begin{definition}
		Let $F,G \in \D^b(\kk_\V)$ and let $\varepsilon \geq  0$. The sheaves $F$ and $G$ are $\varepsilon$-isomorphic if there are morphisms $f : K_\varepsilon  \star F \longrightarrow G$ and $g : K_\varepsilon \star G \longrightarrow F$ such that the  diagrams
		
		\begin{align*}
			\xymatrix{ K_{2\varepsilon} \star F \ar[rr]^-{{K_{2\varepsilon}} \star f} \ar@/_2pc/[rrrr]_{{\chi_{2\varepsilon,0}} \star F} && K_{\varepsilon} \star G \ar[rr]^-{g} &&   F
			},\\
			\xymatrix{ K_{2\varepsilon} \star G \ar[rr]^-{{K_{2\varepsilon}} \star g} \ar@/_2pc/[rrrr]_{{\chi_{2\varepsilon,0}} \star G} && K_{\varepsilon} \star F \ar[rr]^-{f} &&   G
			}.
		\end{align*}
		are commutative. The pair of morphisms $(f,g)$ is called a pair of $\varepsilon$-isomorphisms.	
	\end{definition}

\begin{definition}For $F,G \in \D^b(\kk_\V) $, their \textit{convolution distance} is 
	\begin{equation*}
		d_C(F,G) := \inf(\{\varepsilon \geq 0 \mid F ~\textnormal{and}~G~\textnormal{are}~\varepsilon -\textnormal{isomorphic}\} \cup \{ \infty \}).
	\end{equation*}

\end{definition}

\begin{definition}
A pseudo-extended metric on a set $X$ is a map $\delta : X \times X \longrightarrow \R_{\geq 0}\cup \{+\infty\}$ satisfying for all $x,y,z \in X$: $\delta(x,y) \leq \delta (x,z) + \delta (z,y)$. 
\end{definition}

It is proved in \cite{KS18} that the convolution is, indeed, a pseudo-extended metric, that is, it satisfies the triangular inequality. Having isomorphic global sections is a necessary condition for two sheaves to be at finite convolution distance, as expressed in the following proposition, which can be found as \cite[Remark 2.5 (i)]{KS18}. 

\begin{proposition}\label{p:sections}Let $F,G \in \D^b(\kk_\V)$ such that $d_C(F,G) < +\infty$. Then: \[\Rr \Gamma (\V ; F) \simeq \Rr \Gamma (\V;G).\]

\end{proposition}

Moreover, it satisfies the following important stability property.

\begin{theorem} \label{th:stability}
Let $u,v : X \longrightarrow \V$ be continuous maps, and let $F \in \D^b(\kk_\V)$. Then,

\[d_C(\Rr u_\ast F, \Rr v_\ast F) \leq \sup_{x \in X} \|u(x) - v(x)\|.\]
\end{theorem}

We will often make use of the following result, that we call the additivity of interleavings, which is a direct consequence of the additivity of the convolution functor.

\begin{proposition}[Additivity of interleavings] \label{p:additivityinterleavings}
Let $(F_i)_{i \in I}$ and $(G_j)_{j\in J}$ be two finite families of $\D^b(\kk_\V)$. For all $I' \subseteq I$ and $J' \subseteq J$ of the same cardinality (eventually empty), and for all bijections $\sigma : I' \longrightarrow J'$, one has: 

\[d_C(\oplus_{i\in i} F_i, \oplus_{j\in J} G_j) \leq \max \left ( \max_{i \in I'} d_C(F_i, G_{\sigma(i)}), \max_{i\in I\backslash I'}d_C(F_i,0),  \max_{j\in J\backslash J'}d_C(G_j,0)) \right).\]
\end{proposition}

\begin{proof}
Let $I' \subseteq I$ and $J' \subseteq J$ of the same cardinality (eventually empty), and  $\sigma : I' \longrightarrow J'$ a bijection. We set \[M = \max \left ( \max_{i \in I'} d_C(F_i, G_{\sigma(i)}), \max_{i\in I\backslash I'}d_C(F_i,0),  \max_{j\in J\backslash J'}d_C(G_j,0)) \right).\] If $M = + \infty$, the inequality is true. Let us now assume that $M < + \infty$. Let $\varepsilon > M$. Then for all $i \in I \backslash I'$, $F_i$ is $\varepsilon$-interleaved with $0$, so the canonical map $F_i \star K_{2\varepsilon} \longrightarrow F_i$ is zero. Similarly,  for all $j \in J\backslash J'$, $G_j$ is $\varepsilon$-interleaved with $0$, so the canonical map $G_j \star K_{2\varepsilon} \longrightarrow G_j$ is zero. Moreover, for all $i \in I'$, there exists a pair of $\varepsilon$-interleavings morphisms $f_i : F_i \star K_\varepsilon \longrightarrow G_{\sigma(i)}$ and $g_i : G_{\sigma(i)} \star K_\varepsilon \longrightarrow F_i$. 

Since $(\oplus_{i \in I} F_i) \star K_\varepsilon \simeq \oplus_{i \in I} (F_i \star K_\varepsilon)$ and $(\oplus_{j \in J} G_j) \star K_\varepsilon \simeq \oplus_{j \in J} (G_j \star K_\varepsilon)$, we can define $f : \oplus_{i \in I} F_i \star K_\varepsilon \longrightarrow \oplus_{j \in J} G_{j}$ and $g : \oplus_{j \in J} G_j \star K_\varepsilon \longrightarrow \oplus_{i \in I} F_i$ by:

\[\forall i \in I', f_{|F_i \star K_\varepsilon} = f_i ~~~\textnormal{and} ~~~g_{|G_{\sigma(i) }\star K_\varepsilon} = g_{i}, \]

\[\forall i \in I \backslash I', f_{|F_i \star K_\varepsilon} = 0,~~~\textnormal{and} ~~~ \forall j \in J \backslash J', g_{|G_j \star K_\varepsilon} = 0.\]

\noindent Let us verify that $(f,g)$ is an $\varepsilon$-interleaving pair between $\oplus_{i\in i} F_i$ and $\oplus_{j\in J} G_j$. For all $i \in I'$, one has:

\begin{align*}
(g \circ f \star K_\varepsilon)_{|F_i \star K_{2\varepsilon}} &= g \circ f_i\star K_\varepsilon \\
&= g_{i} \circ f_i\star K_\varepsilon \\ 
&= \chi_{2\varepsilon,0}\star F_i \quad \textnormal{($(f_i,g_i)$ is an $\varepsilon$-interleaving pair)}\\
&= (\chi_{2\varepsilon,0}\star F)_{|F_i \star K_{2\varepsilon}}.
\end{align*}    

\noindent Also for all $i\in I \backslash I'$, one has:

\begin{align*}
(g \circ f \star K_\varepsilon)_{|F_i \star K_{2\varepsilon}} &= 0 \\
&= \chi_{2\varepsilon,0}\star F_i \quad \textnormal{($F_i$ is $\varepsilon$-interleaved with $0$)}\\
&= (\chi_{2\varepsilon,0}\star F)_{|F_i \star K_{2\varepsilon}}.
\end{align*}  

\noindent Therefore, for all $i \in I$, one has $(g \circ f \star K_\varepsilon)_{|F_i \star K_{2\varepsilon}} = (\chi_{2\varepsilon,0}\star F)_{|F_i \star K_{2\varepsilon}}$, which implies that $g \circ f \star K_\varepsilon = \chi_{2\varepsilon,0}\star F$. A similar computation yields $f \circ g \star K_\varepsilon = \chi_{2\varepsilon,0}\star G$. Thus, $(f,g)$ is indeed an $\varepsilon$-interleaving pair.

By taking the infimum over $\varepsilon > M$, we get the desired inequality.

\end{proof}

\subsection{PL-sheaves and functions}

We consider a finite dimensional real vector space $\V$ endowed with a norm $\|\cdot\|$. We equip $\V$ with the topology induced by the norm $\|\cdot\|$, and $\D^b(\kk_\V)$ with the convolution distance $d_C$ associated to $\|\cdot\|$. The notion of Piecewise-Linear sheaves was introduced by Kashiwara-Schapira in \cite{KS18}. 

\begin{definition}
A convex polytope $P$ in $\V$, is the intersection of a finite family of open or closed affine half-spaces.
\end{definition}

\begin{definition}
A sheaf $F \in \D_{\R c}^b(\kk_\V)$ is Piecewise Linear (PL) if there exists a locally-finite family $(P_a)_{a\in A}$ of locally closed convex polytopes covering $\V$, such that $F_{|P_a}$ is locally constant and of finite rank for all $a\in A$.
\end{definition}

We shall denote by $\D^b_{\textnormal{PL}}(\kk_\V)$ the full subcategory of $\D^b(\kk_\V)$ consisting of PL sheaves. The following approximation theorem is proved in \cite{KS18}.

\begin{theorem}\label{th:PLapprox}
Let $F \in \D^b_{\R c}(\kk_\V)$ and $C \in \Z_{\geq 0}$ such that for all $|i| > C$, one has $\Ho^i(F) \simeq 0$. Then for any  $\varepsilon > 0$, there exists a sheaf $F_\varepsilon \in \D^b_{\textnormal{PL}}(\kk_\V)$ satisfying:

\begin{enumerate}
    \item $d_C(F, F_\varepsilon) \leq \varepsilon$,
    
    \item $\textnormal{supp}(F_\varepsilon) \subset \textnormal{supp}(F) + B(0,\varepsilon)$,
    
    \item $\Ho^i(F_\varepsilon) \simeq 0$, for all $|i| > C + \mathrm{dim}(\V) + 1$.
\end{enumerate}
\end{theorem}

\begin{proof}
$(1)$ and $(2)$ are \cite[Theorem 2.11]{KS18}. For $(3)$, we have to use the construction of the proof of \cite[Theorem 2.11]{KS18}. More precisely, the authors construct a simplicial complex $(S, \Delta)$ such that there is an homeomorphism $f : |S| \overset{\sim}{\longrightarrow} \V$ and a PL continuous map $g : |S| \longrightarrow \V$, such that $F \simeq \Rr f_\ast f^{-1} F $ and $F_{\varepsilon} \simeq \Rr g_\ast f^{-1} F$. We conclude by observing that the flabby dimension of $\V$ (hence of $|S|$) is $\dim(\V) + 1$ \cite[Exercise III.2]{KS90}.
\end{proof}

Following \cite{Leb21}, we introduce the PL counterpart of constructible functions.

\begin{definition}
A function $\varphi : \V \longrightarrow \Z$ is PL-constructible, if there exists a locally-finite covering $\V = \bigcup_{a\in A} P_a$ by locally closed convex polytopes, such that $\varphi$ is constant on each $P_a$.
\end{definition}

We denote by $\CF_\textnormal{PL}(\V)$ the group of PL-constructible functions on $\V$. 

\begin{proposition}[\cite{Leb21}]
    Any $\varphi \in \CF_\textnormal{PL}(\V)$ with compact support can be written as a finite sum $\varphi = \sum_\lambda C_\lambda \cdot 1_{    X_\lambda}$, where $X_\lambda$ is a compact convex polytope, and $C_\lambda \in \Z$.
\end{proposition}

\section{Main result}

Let $(\V, \| \cdot \|)$ be a finite dimensional normed real vector space. We endow $\D^b(\kk_\V)$ with the associated convolution distance $d_C$ \cite{KS18}. 

\begin{definition}
Let $\mathscr{C}$ be an abelian category. A sequence of objects $(X_n)_{n \geq 0}$ of $\D^b(\mathscr{C})$ is said to be uniformly bounded, if there exists an integer $C \geq 0$ such that for all $|i| > C$, one has for all $n \geq 0$, $\Ho^i (X_n) \simeq 0$.
\end{definition}

Let $\delta$ be a pseudo-extended metric on $\CF (\V)$.

\begin{definition}

The pseudo-extended metric $\delta$ is said to be $d_C$-\emph{dominated} if for all uniformly bounded sequences $(F_n) \in \D^b_{\R c}(\kk_\V)$ of compactly supported sheaves, and $F \in \D^b_{\R c}(\kk_\V)$ with compact support, one has:

\[d_C(F,F_n) \underset{n \longrightarrow +\infty}{\longrightarrow} 0  ~~\Longrightarrow \delta(\chi(F),\chi(F_n)) ~~ \underset{n \longrightarrow +\infty}{\longrightarrow} 0. \]

\end{definition}

It shall be noted that by Proposition \ref{p:sections} and Lemma \ref{lem:integralsection}, the condition $d_C(F,G) < +\infty$  implies that $\int \chi(F) \dk= \int \chi(G) \dk$. Our aim is to characterize all $d_C$-dominated pseudo-extended metrics on $\CF (\V)$. This will be achieved in Theorem \ref{th:main}. In all this section, $\delta$ designates a $d_C$-dominated pseudo-extended metric on $\CF (\V)$.

Our strategy is to prove that for any $\varphi \in \CF(\V)$ with compact support, it is possible to concentrate the "mass" of $\varphi$ on one single point, that is $\delta(\varphi, (\int \varphi \dk)\cdot 1_{\{0\}}) = 0$. To do so, we first assume that $\varphi$ is PL-constructible, which allows us to use rather straightforward arguments instead of sophisticated one from subanalytic geometry. We then generalize to any stratification thanks to Kashiwara-Schapira's approximation Theorem \ref{th:PLapprox}.

In Section \ref{section:flag}, we introduce the notion of $\varepsilon$-flag, which is a nested sequence of convex compact sets, and allows us to successively concentrate the mass of an indicator PL-function onto one single point. This allows us to treat the PL-case in Section \ref{section:PLcase}, and the general one in Section \ref{section:generalcase}

\subsection{Convolution distance of the difference of compact convex subsets} \label{section:flag}

Recall that for $x\in \V$ and $\varepsilon \geq 0$, we denote by $B(x,\varepsilon)$ the closed ball of radius $\varepsilon$ centered at $x$.
\begin{lemma}\label{lem:interleaving}
Let $F \in \D^b(\kk_\V)$ with compact support, and $\varepsilon \geq 0$.  If for all $x \in \textnormal{supp}(F)$ one has $\Rr \Gamma (B(x, \varepsilon) ; F) \simeq 0$, then  $F$ is $\frac{\varepsilon}{2}$-isomorphic to $0$.
\end{lemma}

\begin{proof}
Let $F$ and $\varepsilon$ as in the statement. By definition of interleavings, it is sufficient to prove that the canonical map $F \star K_\varepsilon \longrightarrow F$ is zero. Let $x \in \V$. If $x \not \in \textnormal{supp}(F)$, it is clear that the induced morphism $(F \star K_\varepsilon)_x \longrightarrow F_x$ is zero. Let us assume that $x \in \textnormal{supp}(F)$. By equation (2.12) in \cite{PS20}, one has \[(F \star K_\varepsilon)_x \simeq \Rr \Gamma(B(x,\varepsilon);F) \simeq 0.\]

Therefore the morphism $(F \star K_\varepsilon)_x \longrightarrow F_x$ is zero in every case, which implies that $F \star K_\varepsilon \longrightarrow F$ is also zero.
\end{proof}

\begin{definition}
Given $X \subset \V$ and $\varepsilon \geq 0$, the $\varepsilon$-thickening of $X$ is defined by: 

\[T_\varepsilon(X) := \{v \in \V \mid d(v,X) \leq \varepsilon\}.\]
\end{definition}

\begin{lemma} \label{lemma:vanishing}
Let $X \subset Y$ be compact convex subsets of $\V$, and assume that there exists $\varepsilon\geq 0$ such that $Y \subset T_\varepsilon(X)$. Then $d_C(\kk_{Y \backslash X},0) \leq \frac{\varepsilon}{2}$.
\end{lemma}

\begin{proof}
For $y\in Y$ and $\varepsilon' > \varepsilon$,  one has the following distinguished triangle: 

\[\Rr \Gamma (B(y,\varepsilon') ; \kk_Y) \longrightarrow \Rr \Gamma (B(y,\varepsilon') ; \kk_X) \longrightarrow \Rr \Gamma (B(y,\varepsilon') ; \kk_{Y\backslash X}) \stackrel{+1}{\longrightarrow}.\]

\noindent By hypothesis,  $B(y, \varepsilon')\cap Y \cap X$ is non-empty and convex. Since $X$ and $Y$ are closed convex subsets, we deduce that the map $\Rr \Gamma (B(y,\varepsilon') ; \kk_Y) \longrightarrow \Rr \Gamma (B(y,\varepsilon') ; \kk_X)$ is an isomorphism. Therefore,  $\Rr \Gamma (B(y,\varepsilon') ; \kk_{Y\backslash X}) \simeq 0$, for all $y \in \textnormal{supp}(\kk_{Y\backslash X}) \subset Y$ and $\varepsilon' > \varepsilon$.  Lemma \ref{lem:interleaving} implies that $d_C(\kk_{Y\backslash X},0) \leq \frac{\varepsilon}{2}$.
\end{proof}

\begin{definition}
Let $\varepsilon \geq 0$. An $\varepsilon$-flag is a sequence of nested subsets $X^0 \subset X^1 \subset ...\subset X^n$ of $\V$ satisfying: 

\begin{enumerate}
    \item $X^i$ is a compact convex subset of $\V$, for all $i$;
    \item $X^0 = \{x_0\}$ is a single point;
    \item $X^i \subset T_\varepsilon(X^{i-1})$ for all $i$.
\end{enumerate}
We designate this data by $X^\bullet$.
\end{definition}

\noindent Given an $\varepsilon$-flag $X^\bullet = (X^i)_{i = 0... n}$, and $i \in \llbracket 0,n \rrbracket$, we define the spaces $\Gr_i(X^\bullet)$ by

\[\Gr_0(X^\bullet) := X^0, ~~~\textnormal{and for all~} i\geq 1, \Gr_i(X^\bullet) := X^i \backslash X^{i-1}. \]

\noindent It is immediate to verify that $\Gr_i(X^\bullet)$ is  locally closed for all $i \in \llbracket0,n \rrbracket$, and that one has $X^n = \sqcup_i \Gr_i(X^\bullet).$ Moreover, we set:

\[S(X^\bullet) := \bigoplus_{i = 0}^n \kk_{\Gr_i(X^\bullet)} \in \D^b_{\R c}(\kk_\V).\]

\begin{proposition}\label{P:flag}
    Let $X^\bullet = (X^i)_{i = 0 ... n }$ be an $\varepsilon$-flag. Then one has:
    
    \begin{enumerate}
        \item $\chi(S(X^\bullet)) = \chi(\kk_{X^n})$;
        \item $d_C(S(X^\bullet), \kk_{X^0} ) \leq \frac{\varepsilon}{2}$.
    \end{enumerate}
\end{proposition}

\begin{proof}
\begin{enumerate}
    \item This is a direct consequence of the fact that $X^n = \sqcup_i \Gr_i(X^\bullet).$
    
    \item For $i \geq 1$, the definition of $\varepsilon$-flag implies that the pair $(X^{i-1}, X^i)$ satisfy the hypothesis of Lemma \ref{lemma:vanishing}. Therefore, $d_C(\kk_{\Gr_i(X^\bullet)}, 0) \leq \frac{\varepsilon}{2}$. By additivity of interleavings, one deduces:
    
    \begin{align*}
        d_C(S(X^\bullet), \kk_{X^0} ) &= d_C(\kk_{X^0} \oplus \bigoplus_{i=1}^n \kk_{\Gr_i(X^\bullet)}, \kk_{X^0}) \\
        &\leq \max \left (d_C \left (\kk_{X^0},\kk_{X^0} \right), \max_{i = 1 ... n} d_C \left (\kk_{\Gr_i(X^\bullet)}, 0 \right ) \right ) \quad \textnormal{(Proposition \ref{p:additivityinterleavings})}\\
        &= \max_{i = 1 ... n} d_C(\kk_{\Gr_i(X^\bullet)}, 0) \\
        &\leq \frac{\varepsilon}{2}.
    \end{align*}
\end{enumerate}
\end{proof}

\subsection{PL-case}\label{section:PLcase}

The first step of our proof is the following concentration lemma in the Piecewise-Linear (PL) case, that we will extend later on to arbitrary stratification by density of PL-sheaves with respect to the convolution distance.

\begin{lemma} \label{lemma:reunion} 
Let $\varphi \in \CF_\PL(\V)$ with compact support, such that $\varphi = \sum_{\lambda \in A} C_\lambda \cdot 1_{X_\lambda}$, with $A$ finite and $X_\lambda$ compact and convex polytopes. For $\lambda \in A$, let $x_\lambda \in X_\lambda$. Then one has  
\[\delta \left (\varphi,\sum_{\lambda \in A} C_\lambda \cdot 1_{\{x_\lambda\}} \right) = 0.\]
\end{lemma}

\begin{proof}
 We consider the linear deformation retraction $H_\lambda : X_\lambda \times [0,1] \longrightarrow X_\lambda$ from  $\{x_\lambda\}$ to $X_\lambda$  defined by: 

\[ H_\lambda(x,t) = (1-t)\cdot x_\lambda + t\cdot x. \]

We set $\ell_\lambda = \max\{\|x - x_\lambda\| \mid x \in X_\lambda \}$ and $\ell = \max_\lambda \ell_\lambda$.  Let  $\varepsilon>0$ and $n = \lceil\frac{\ell}{\varepsilon}\rceil$. We define for $i \in \llbracket 0, n \rrbracket$ the sequence of subsets $X_\lambda^i := H_\lambda (X_\lambda \times {[0, \frac{i}{n}]})$. By construction, $X_\lambda^\bullet = (X_\lambda^i)_{i=0...n}$ is an $\varepsilon$-flag. We depict an illustration of $X_\lambda^\bullet$ in figure \ref{fig:flag}.

\begin{figure}
    \centering
   
 \resizebox{1\textwidth}{!}{%

\definecolor{zzwwff}{rgb}{0.6,0.4,1}
\definecolor{ccwwff}{rgb}{0.8,0.4,1}
\definecolor{cczzff}{rgb}{0.8,0.6,1}
\definecolor{ffffff}{rgb}{1,1,1}
\definecolor{ccccff}{rgb}{0.8,0.8,1}
\definecolor{zzttqq}{rgb}{0.6,0.2,0}
\begin{tikzpicture}[line cap=round,line join=round,>=triangle 45,x=1cm,y=1cm]
\clip(-19.346689233632603,-0.8058529748790468) rectangle (9.498837943338705,19.137882481885843);
\fill[line width=2pt,color=zzttqq] (-5,0) -- (2.7922362240260115,4.62237517933541) -- (2.6852600049119006,13.681837291298972) -- (-5.213952438228222,18.118924223927124) -- (-13.006188662254232,13.496549044591717) -- (-12.899212443140126,4.437086932628156) -- cycle;
\fill[line width=2pt,color=ccccff,fill=ccccff,fill opacity=1] (-13.006188662254232,13.496549044591717) -- (-10.071986594541027,13.13286527126167) -- (-9.99296294358437,6.440613687503346) -- (-12.899212443140126,4.437086932628156) -- cycle;
\fill[line width=2pt,color=ccccff,fill=ccccff,fill opacity=1] (-13.006188662254232,13.496549044591717) -- (-5.213952438228222,18.118924223927124) -- (-4.331081456037921,16.57409921608025) -- (-10.071986594541027,13.13286527126167) -- cycle;
\fill[line width=2pt,color=ccccff,fill=ccccff,fill opacity=1] (-5.213952438228222,18.118924223927124) -- (2.6852600049119006,13.681837291298972) -- (1.539121440735141,13.276732452920562) -- (-4.331081456037921,16.57409921608025) -- cycle;
\fill[line width=2pt,color=ccccff,fill=ccccff,fill opacity=1] (2.6852600049119006,13.681837291298972) -- (2.7922362240260115,4.62237517933541) -- (1.6186195066309166,6.544304237366243) -- (1.539121440735141,13.276732452920562) -- cycle;
\fill[line width=2pt,color=ccccff,fill=ccccff,fill opacity=1] (2.7922362240260115,4.62237517933541) -- (-5,0) -- (-4.172085324246375,3.1092427849716113) -- (1.6186195066309166,6.544304237366243) -- cycle;
\fill[line width=2pt,color=ccccff,fill=ccccff,fill opacity=1] (-5,0) -- (-4.172085324246375,3.1092427849716113) -- (-9.99296294358437,6.440613687503346) -- (-12.899212443140126,4.437086932628156) -- cycle;
\fill[line width=2pt,color=cczzff,fill=cczzff,fill opacity=1] (-4.172085324246375,3.1092427849716113) -- (-3.3248203998271095,6.291155644303581) -- (-7.086713444028617,8.444140442378535) -- (-9.99296294358437,6.440613687503346) -- cycle;
\fill[line width=2pt,color=cczzff,fill=cczzff,fill opacity=1] (-9.99296294358437,6.440613687503346) -- (-7.086713444028617,8.444140442378535) -- (-7.137784526827822,12.76918149793162) -- (-10.071986594541027,13.13286527126167) -- cycle;
\fill[line width=2pt,color=cczzff,fill=cczzff,fill opacity=1] (-7.137784526827822,12.76918149793162) -- (-3.4275757713289052,14.993168134839374) -- (-4.331081456037921,16.57409921608025) -- (-10.071986594541027,13.13286527126167) -- cycle;
\fill[line width=2pt,color=cczzff,fill=cczzff,fill opacity=1] (-3.4275757713289052,14.993168134839374) -- (0.36619501197328413,12.862159393253362) -- (1.539121440735141,13.276732452920562) -- (-4.331081456037921,16.57409921608025) -- cycle;
\fill[line width=2pt,color=cczzff,fill=cczzff,fill opacity=1] (0.41757269772418354,8.51115314798547) -- (1.6186195066309166,6.544304237366243) -- (1.539121440735141,13.276732452920562) -- (0.36619501197328413,12.862159393253362) -- cycle;
\fill[line width=2pt,color=cczzff,fill=cczzff,fill opacity=1] (-4.172085324246375,3.1092427849716113) -- (1.6186195066309166,6.544304237366243) -- (0.41757269772418354,8.51115314798547) -- (-3.3248203998271095,6.291155644303581) -- cycle;
\fill[line width=2pt,color=ccwwff,fill=ccwwff,fill opacity=1] (-3.3248203998271095,6.291155644303581) -- (0.41757269772418354,8.51115314798547) -- (-0.7834741111825433,10.478002058604687) -- (-2.477555475407845,9.473068503635549) -- cycle;
\fill[line width=2pt,color=ccwwff,fill=ccwwff,fill opacity=1] (-3.3248203998271095,6.291155644303581) -- (-2.477555475407845,9.473068503635549) -- (-4.180463944472864,10.447667197253722) -- (-7.086713444028617,8.444140442378535) -- cycle;
\fill[line width=2pt,color=ccwwff,fill=ccwwff,fill opacity=1] (-7.086713444028617,8.444140442378535) -- (-4.180463944472864,10.447667197253722) -- (-4.203582459114618,12.405497724601574) -- (-7.137784526827822,12.76918149793162) -- cycle;
\fill[line width=2pt,color=ccwwff,fill=ccwwff,fill opacity=1] (-4.203582459114618,12.405497724601574) -- (-2.524070086619891,13.412237053598506) -- (-3.4275757713289052,14.993168134839374) -- (-7.137784526827822,12.76918149793162) -- cycle;
\fill[line width=2pt,color=ccwwff,fill=ccwwff,fill opacity=1] (-2.524070086619891,13.412237053598506) -- (-0.8067314167885665,12.44758633358616) -- (0.36619501197328413,12.862159393253362) -- (-3.4275757713289052,14.993168134839374) -- cycle;
\fill[line width=2pt,color=ccwwff,fill=ccwwff,fill opacity=1] (-0.7834741111825433,10.478002058604687) -- (0.41757269772418354,8.51115314798547) -- (0.36619501197328413,12.862159393253362) -- (-0.8067314167885665,12.44758633358616) -- cycle;
\fill[line width=2pt,color=zzwwff,fill=zzwwff,fill opacity=1] (-2.477555475407845,9.473068503635549) -- (-0.7834741111825433,10.478002058604687) -- (-1.7768146185999911,12.104708595268635) -- cycle;
\fill[line width=2pt,color=zzwwff,fill=zzwwff,fill opacity=1] (-2.477555475407845,9.473068503635549) -- (-1.7768146185999911,12.104708595268635) -- (-4.180463944472864,10.447667197253722) -- cycle;
\fill[line width=2pt,color=zzwwff,fill=zzwwff,fill opacity=1] (-4.180463944472864,10.447667197253722) -- (-1.7768146185999911,12.104708595268635) -- (-4.203582459114618,12.405497724601574) -- cycle;
\fill[line width=2pt,color=zzwwff,fill=zzwwff,fill opacity=1] (-1.7768146185999911,12.104708595268635) -- (-2.524070086619891,13.412237053598506) -- (-4.203582459114618,12.405497724601574) -- cycle;
\fill[line width=2pt,color=zzwwff,fill=zzwwff,fill opacity=1] (-1.7768146185999911,12.104708595268635) -- (-0.8067314167885665,12.44758633358616) -- (-2.524070086619891,13.412237053598506) -- cycle;
\fill[line width=2pt,color=zzwwff,fill=zzwwff,fill opacity=1] (-1.7768146185999911,12.104708595268635) -- (-0.7834741111825433,10.478002058604687) -- (-0.8067314167885665,12.44758633358616) -- cycle;
\draw [line width=2pt,color=ccccff] (-5,0)-- (2.7922362240260115,4.62237517933541);
\draw [line width=2pt,color=ccccff] (2.7922362240260115,4.62237517933541)-- (2.6852600049119006,13.681837291298972);
\draw [line width=2pt,color=ccccff] (2.6852600049119006,13.681837291298972)-- (-5.213952438228222,18.118924223927124);
\draw [line width=2pt,color=ccccff] (-5.213952438228222,18.118924223927124)-- (-13.006188662254232,13.496549044591717);
\draw [line width=2pt,color=ccccff] (-13.006188662254232,13.496549044591717)-- (-12.899212443140126,4.437086932628156);
\draw [line width=2pt,color=ccccff] (-12.899212443140126,4.437086932628156)-- (-5,0);
\draw [line width=2pt] (-9.99296294358437,6.440613687503346)-- (-4.172085324246375,3.1092427849716113);
\draw [line width=2pt] (-9.99296294358437,6.440613687503346)-- (-10.071986594541027,13.13286527126167);
\draw [line width=2pt] (-10.071986594541027,13.13286527126167)-- (-4.331081456037921,16.57409921608025);
\draw [line width=2pt] (-4.331081456037921,16.57409921608025)-- (1.539121440735141,13.276732452920562);
\draw [line width=2pt] (1.539121440735141,13.276732452920562)-- (1.6186195066309166,6.544304237366243);
\draw [line width=2pt] (1.6186195066309166,6.544304237366243)-- (-4.172085324246375,3.1092427849716113);
\draw [line width=2pt] (-7.086713444028617,8.444140442378535)-- (-7.137784526827822,12.76918149793162);
\draw [line width=2pt] (-7.137784526827822,12.76918149793162)-- (-3.4275757713289052,14.993168134839374);
\draw [line width=2pt] (-3.4275757713289052,14.993168134839374)-- (0.36619501197328413,12.862159393253362);
\draw [line width=2pt] (0.36619501197328413,12.862159393253362)-- (0.41757269772418354,8.51115314798547);
\draw [line width=2pt] (0.41757269772418354,8.51115314798547)-- (-3.3248203998271095,6.291155644303581);
\draw [line width=2pt] (-3.3248203998271095,6.291155644303581)-- (-7.086713444028617,8.444140442378535);
\draw [line width=2pt] (-4.180463944472864,10.447667197253722)-- (-4.203582459114618,12.405497724601574);
\draw [line width=2pt] (-4.203582459114618,12.405497724601574)-- (-2.524070086619891,13.412237053598506);
\draw [line width=2pt] (-2.524070086619891,13.412237053598506)-- (-0.8067314167885665,12.44758633358616);
\draw [line width=2pt] (-0.8067314167885665,12.44758633358616)-- (-0.7834741111825433,10.478002058604687);
\draw [line width=2pt] (-0.7834741111825433,10.478002058604687)-- (-2.477555475407845,9.473068503635549);
\draw [line width=2pt] (-2.477555475407845,9.473068503635549)-- (-4.180463944472864,10.447667197253722);
\draw [line width=2pt,color=ccccff] (-13.006188662254232,13.496549044591717)-- (-10.071986594541027,13.13286527126167);
\draw [line width=2pt,color=ccccff] (-10.071986594541027,13.13286527126167)-- (-9.99296294358437,6.440613687503346);
\draw [line width=2pt,color=ccccff] (-9.99296294358437,6.440613687503346)-- (-12.899212443140126,4.437086932628156);
\draw [line width=2pt,color=ccccff] (-12.899212443140126,4.437086932628156)-- (-13.006188662254232,13.496549044591717);
\draw [line width=2pt,color=ccccff] (-13.006188662254232,13.496549044591717)-- (-5.213952438228222,18.118924223927124);
\draw [line width=2pt,color=ccccff] (-5.213952438228222,18.118924223927124)-- (-4.331081456037921,16.57409921608025);
\draw [line width=2pt,color=ccccff] (-4.331081456037921,16.57409921608025)-- (-10.071986594541027,13.13286527126167);
\draw [line width=2pt,color=ccccff] (-10.071986594541027,13.13286527126167)-- (-13.006188662254232,13.496549044591717);
\draw [line width=2pt,color=ccccff] (-5.213952438228222,18.118924223927124)-- (2.6852600049119006,13.681837291298972);
\draw [line width=2pt,color=ccccff] (2.6852600049119006,13.681837291298972)-- (1.539121440735141,13.276732452920562);
\draw [line width=2pt,color=ccccff] (1.539121440735141,13.276732452920562)-- (-4.331081456037921,16.57409921608025);
\draw [line width=2pt,color=ccccff] (-4.331081456037921,16.57409921608025)-- (-5.213952438228222,18.118924223927124);
\draw [line width=2pt,color=ccccff] (2.6852600049119006,13.681837291298972)-- (2.7922362240260115,4.62237517933541);
\draw [line width=2pt,color=ccccff] (2.7922362240260115,4.62237517933541)-- (1.6186195066309166,6.544304237366243);
\draw [line width=2pt,color=ccccff] (1.6186195066309166,6.544304237366243)-- (1.539121440735141,13.276732452920562);
\draw [line width=2pt,color=ccccff] (1.539121440735141,13.276732452920562)-- (2.6852600049119006,13.681837291298972);
\draw [line width=2pt,color=ccccff] (2.7922362240260115,4.62237517933541)-- (-5,0);
\draw [line width=2pt,color=ccccff] (-5,0)-- (-4.172085324246375,3.1092427849716113);
\draw [line width=2pt,color=ccccff] (-4.172085324246375,3.1092427849716113)-- (1.6186195066309166,6.544304237366243);
\draw [line width=2pt,color=ccccff] (1.6186195066309166,6.544304237366243)-- (2.7922362240260115,4.62237517933541);
\draw [line width=2pt,color=ccccff] (-5,0)-- (-4.172085324246375,3.1092427849716113);
\draw [line width=2pt,color=ccccff] (-4.172085324246375,3.1092427849716113)-- (-9.99296294358437,6.440613687503346);
\draw [line width=2pt,color=ccccff] (-9.99296294358437,6.440613687503346)-- (-12.899212443140126,4.437086932628156);
\draw [line width=2pt,color=ccccff] (-12.899212443140126,4.437086932628156)-- (-5,0);
\draw [line width=2pt,color=cczzff] (-4.172085324246375,3.1092427849716113)-- (-3.3248203998271095,6.291155644303581);
\draw [line width=2pt,color=cczzff] (-3.3248203998271095,6.291155644303581)-- (-7.086713444028617,8.444140442378535);
\draw [line width=2pt,color=cczzff] (-7.086713444028617,8.444140442378535)-- (-9.99296294358437,6.440613687503346);
\draw [line width=2pt,color=cczzff] (-9.99296294358437,6.440613687503346)-- (-4.172085324246375,3.1092427849716113);
\draw [line width=2pt,color=cczzff] (-9.99296294358437,6.440613687503346)-- (-7.086713444028617,8.444140442378535);
\draw [line width=2pt,color=cczzff] (-7.086713444028617,8.444140442378535)-- (-7.137784526827822,12.76918149793162);
\draw [line width=2pt,color=cczzff] (-7.137784526827822,12.76918149793162)-- (-10.071986594541027,13.13286527126167);
\draw [line width=2pt,color=cczzff] (-10.071986594541027,13.13286527126167)-- (-9.99296294358437,6.440613687503346);
\draw [line width=2pt,color=cczzff] (-7.137784526827822,12.76918149793162)-- (-3.4275757713289052,14.993168134839374);
\draw [line width=2pt,color=cczzff] (-3.4275757713289052,14.993168134839374)-- (-4.331081456037921,16.57409921608025);
\draw [line width=2pt,color=cczzff] (-4.331081456037921,16.57409921608025)-- (-10.071986594541027,13.13286527126167);
\draw [line width=2pt,color=cczzff] (-10.071986594541027,13.13286527126167)-- (-7.137784526827822,12.76918149793162);
\draw [line width=2pt,color=cczzff] (-3.4275757713289052,14.993168134839374)-- (0.36619501197328413,12.862159393253362);
\draw [line width=2pt,color=cczzff] (0.36619501197328413,12.862159393253362)-- (1.539121440735141,13.276732452920562);
\draw [line width=2pt,color=cczzff] (1.539121440735141,13.276732452920562)-- (-4.331081456037921,16.57409921608025);
\draw [line width=2pt,color=cczzff] (-4.331081456037921,16.57409921608025)-- (-3.4275757713289052,14.993168134839374);
\draw [line width=2pt,color=cczzff] (0.41757269772418354,8.51115314798547)-- (1.6186195066309166,6.544304237366243);
\draw [line width=2pt,color=cczzff] (1.6186195066309166,6.544304237366243)-- (1.539121440735141,13.276732452920562);
\draw [line width=2pt,color=cczzff] (1.539121440735141,13.276732452920562)-- (0.36619501197328413,12.862159393253362);
\draw [line width=2pt,color=cczzff] (0.36619501197328413,12.862159393253362)-- (0.41757269772418354,8.51115314798547);
\draw [line width=2pt,color=cczzff] (-4.172085324246375,3.1092427849716113)-- (1.6186195066309166,6.544304237366243);
\draw [line width=2pt,color=cczzff] (1.6186195066309166,6.544304237366243)-- (0.41757269772418354,8.51115314798547);
\draw [line width=2pt,color=cczzff] (0.41757269772418354,8.51115314798547)-- (-3.3248203998271095,6.291155644303581);
\draw [line width=2pt,color=cczzff] (-3.3248203998271095,6.291155644303581)-- (-4.172085324246375,3.1092427849716113);
\draw [line width=2pt,color=ccwwff] (-3.3248203998271095,6.291155644303581)-- (0.41757269772418354,8.51115314798547);
\draw [line width=2pt,color=ccwwff] (0.41757269772418354,8.51115314798547)-- (-0.7834741111825433,10.478002058604687);
\draw [line width=2pt,color=ccwwff] (-0.7834741111825433,10.478002058604687)-- (-2.477555475407845,9.473068503635549);
\draw [line width=2pt,color=ccwwff] (-2.477555475407845,9.473068503635549)-- (-3.3248203998271095,6.291155644303581);
\draw [line width=2pt,color=ccwwff] (-3.3248203998271095,6.291155644303581)-- (-2.477555475407845,9.473068503635549);
\draw [line width=2pt,color=ccwwff] (-2.477555475407845,9.473068503635549)-- (-4.180463944472864,10.447667197253722);
\draw [line width=2pt,color=ccwwff] (-4.180463944472864,10.447667197253722)-- (-7.086713444028617,8.444140442378535);
\draw [line width=2pt,color=ccwwff] (-7.086713444028617,8.444140442378535)-- (-3.3248203998271095,6.291155644303581);
\draw [line width=2pt,color=ccwwff] (-7.086713444028617,8.444140442378535)-- (-4.180463944472864,10.447667197253722);
\draw [line width=2pt,color=ccwwff] (-4.180463944472864,10.447667197253722)-- (-4.203582459114618,12.405497724601574);
\draw [line width=2pt,color=ccwwff] (-4.203582459114618,12.405497724601574)-- (-7.137784526827822,12.76918149793162);
\draw [line width=2pt,color=ccwwff] (-7.137784526827822,12.76918149793162)-- (-7.086713444028617,8.444140442378535);
\draw [line width=2pt,color=ccwwff] (-4.203582459114618,12.405497724601574)-- (-2.524070086619891,13.412237053598506);
\draw [line width=2pt,color=ccwwff] (-3.4275757713289052,14.993168134839374)-- (-7.137784526827822,12.76918149793162);
\draw [line width=2pt,color=ccwwff] (-7.137784526827822,12.76918149793162)-- (-4.203582459114618,12.405497724601574);
\draw [line width=2pt,color=ccwwff] (-2.524070086619891,13.412237053598506)-- (-0.8067314167885665,12.44758633358616);
\draw [line width=2pt,color=ccwwff] (-0.8067314167885665,12.44758633358616)-- (0.36619501197328413,12.862159393253362);
\draw [line width=2pt,color=ccwwff] (0.36619501197328413,12.862159393253362)-- (-3.4275757713289052,14.993168134839374);
\draw [line width=2pt,color=ccwwff] (-0.7834741111825433,10.478002058604687)-- (0.41757269772418354,8.51115314798547);
\draw [line width=2pt,color=ccwwff] (0.41757269772418354,8.51115314798547)-- (0.36619501197328413,12.862159393253362);
\draw [line width=2pt,color=ccwwff] (0.36619501197328413,12.862159393253362)-- (-0.8067314167885665,12.44758633358616);
\draw [line width=2pt,color=ccwwff] (-0.8067314167885665,12.44758633358616)-- (-0.7834741111825433,10.478002058604687);
\draw [line width=2pt,color=zzwwff] (-2.477555475407845,9.473068503635549)-- (-0.7834741111825433,10.478002058604687);
\draw [line width=2pt,color=zzwwff] (-0.7834741111825433,10.478002058604687)-- (-1.7768146185999911,12.104708595268635);
\draw [line width=2pt,color=zzwwff] (-1.7768146185999911,12.104708595268635)-- (-2.477555475407845,9.473068503635549);
\draw [line width=2pt,color=zzwwff] (-2.477555475407845,9.473068503635549)-- (-1.7768146185999911,12.104708595268635);
\draw [line width=2pt,color=zzwwff] (-1.7768146185999911,12.104708595268635)-- (-4.180463944472864,10.447667197253722);
\draw [line width=2pt,color=zzwwff] (-4.180463944472864,10.447667197253722)-- (-2.477555475407845,9.473068503635549);
\draw [line width=2pt,color=zzwwff] (-4.180463944472864,10.447667197253722)-- (-1.7768146185999911,12.104708595268635);
\draw [line width=2pt,color=zzwwff] (-1.7768146185999911,12.104708595268635)-- (-4.203582459114618,12.405497724601574);
\draw [line width=2pt,color=zzwwff] (-4.203582459114618,12.405497724601574)-- (-4.180463944472864,10.447667197253722);
\draw [line width=2pt,color=zzwwff] (-2.524070086619891,13.412237053598506)-- (-4.203582459114618,12.405497724601574);
\draw [line width=2pt,color=zzwwff] (-4.203582459114618,12.405497724601574)-- (-1.7768146185999911,12.104708595268635);
\draw [line width=2pt,color=zzwwff] (-1.7768146185999911,12.104708595268635)-- (-0.8067314167885665,12.44758633358616);
\draw [line width=2pt,color=zzwwff] (-0.8067314167885665,12.44758633358616)-- (-2.524070086619891,13.412237053598506);
\draw [line width=2pt,color=zzwwff] (-1.7768146185999911,12.104708595268635)-- (-0.7834741111825433,10.478002058604687);
\draw [line width=2pt,color=zzwwff] (-0.7834741111825433,10.478002058604687)-- (-0.8067314167885665,12.44758633358616);
\draw [line width=2pt,color=zzwwff] (-0.8067314167885665,12.44758633358616)-- (-1.7768146185999911,12.104708595268635);
\draw [line width=2pt,dotted,color=ffffff] (-5.213952438228222,18.118924223927124)-- (-1.7768146185999911,12.104708595268635);
\draw [line width=2pt,dotted,color=ffffff] (-3.3799578800888757,14.909847569669099) -- (-3.7377461363286595,14.97330563160761);
\draw [line width=2pt,dotted,color=ffffff] (-3.3799578800888757,14.909847569669099) -- (-3.253020920499557,15.250327187588148);
\draw [line width=2pt,dotted,color=ffffff] (-3.610809176739339,15.313785249526658) -- (-3.9685974329791227,15.37724331146517);
\draw [line width=2pt,dotted,color=ffffff] (-3.610809176739339,15.313785249526658) -- (-3.4838722171500187,15.65426486744571);
\draw [line width=2pt,dotted,color=ffffff] (-3.14910658343841,14.505909889811537) -- (-3.506894839678194,14.56936795175005);
\draw [line width=2pt,dotted,color=ffffff] (-3.14910658343841,14.505909889811537) -- (-3.022169623849092,14.846389507730589);
\draw [line width=2pt,dotted,color=ffffff] (-13.006188662254232,13.496549044591717)-- (-1.7768146185999911,12.104708595268635);
\draw [line width=2pt,dotted,color=ffffff] (-7.160642969692491,12.772014719557681) -- (-7.425838560874105,12.52359841504865);
\draw [line width=2pt,dotted,color=ffffff] (-7.160642969692491,12.772014719557681) -- (-7.357164719980116,13.0776592248117);
\draw [line width=2pt,dotted,color=ffffff] (-7.622360311161731,12.829242920302669) -- (-7.887555902343345,12.580826615793638);
\draw [line width=2pt,dotted,color=ffffff] (-7.622360311161731,12.829242920302669) -- (-7.818882061449355,13.134887425556688);
\draw [line width=2pt,dotted,color=ffffff] (-6.698925628223253,12.714786518812693) -- (-6.964121219404867,12.46637021430366);
\draw [line width=2pt,dotted,color=ffffff] (-6.698925628223253,12.714786518812693) -- (-6.895447378510878,13.020431024066713);
\draw [line width=2pt,dotted,color=ffffff] (-12.899212443140126,4.437086932628156)-- (-1.7768146185999911,12.104708595268635);
\draw [line width=2pt,dotted,color=ffffff] (-7.146489355588974,8.402931785610923) -- (-7.179572704875014,8.041068753611109);
\draw [line width=2pt,dotted,color=ffffff] (-7.146489355588974,8.402931785610923) -- (-7.496454356865104,8.50072677428568);
\draw [line width=2pt,dotted,color=ffffff] (-7.529537706151144,8.138863742285867) -- (-7.562621055437183,7.777000710286053);
\draw [line width=2pt,dotted,color=ffffff] (-7.529537706151144,8.138863742285867) -- (-7.879502707427273,8.236658730960624);
\draw [line width=2pt,dotted,color=ffffff] (-6.763441005026807,8.66699982893598) -- (-6.7965243543128455,8.305136796936166);
\draw [line width=2pt,dotted,color=ffffff] (-6.763441005026807,8.66699982893598) -- (-7.113406006302935,8.764794817610737);
\draw [line width=2pt,dotted,color=ffffff] (-5,0)-- (-1.7768146185999911,12.104708595268635);
\draw [line width=2pt,dotted,color=ffffff] (-3.3285506101569995,6.277146797867711) -- (-3.118656309019904,5.980526258662725);
\draw [line width=2pt,dotted,color=ffffff] (-3.3285506101569995,6.277146797867711) -- (-3.6581583095800876,6.124182336605907);
\draw [line width=2pt,dotted,color=ffffff] (-3.4482640084429903,5.8275617974009215) -- (-3.238369707305895,5.530941258195936);
\draw [line width=2pt,dotted,color=ffffff] (-3.4482640084429903,5.8275617974009215) -- (-3.7778717078660784,5.674597336139118);
\draw [line width=2pt,dotted,color=ffffff] (-3.208837211871005,6.726731798334502) -- (-2.9989429107339096,6.430111259129516);
\draw [line width=2pt,dotted,color=ffffff] (-3.208837211871005,6.726731798334502) -- (-3.538444911294093,6.573767337072698);
\draw [line width=2pt,dotted,color=ffffff] (2.7922362240260115,4.62237517933541)-- (-1.7768146185999911,12.104708595268635);
\draw [line width=2pt,dotted,color=ffffff] (0.3864757381312806,8.562077908490473) -- (0.7459540281391692,8.50902396480009);
\draw [line width=2pt,dotted,color=ffffff] (0.3864757381312806,8.562077908490473) -- (0.26946757728685294,8.218059809803954);
\draw [line width=2pt,dotted,color=ffffff] (0.6289458672947416,8.165005866113571) -- (0.9884241573026301,8.111951922423188);
\draw [line width=2pt,dotted,color=ffffff] (0.6289458672947416,8.165005866113571) -- (0.5119377064503139,7.820987767427052);
\draw [line width=2pt,dotted,color=ffffff] (0.1440056089678197,8.959149950867376) -- (0.5034838989757082,8.906096007176993);
\draw [line width=2pt,dotted,color=ffffff] (0.1440056089678197,8.959149950867376) -- (0.026997448123392036,8.615131852180856);
\draw [line width=2pt,dotted,color=ffffff] (2.6852600049119006,13.681837291298972)-- (-1.7768146185999911,12.104708595268635);
\draw [line width=2pt,dotted,color=ffffff] (0.23489453766357687,12.815750992832381) -- (0.3611963526142427,13.156466729874642);
\draw [line width=2pt,dotted,color=ffffff] (0.23489453766357687,12.815750992832381) -- (0.5472490336976698,12.630079156692965);
\draw [line width=2pt,dotted,color=ffffff] (0.673550848648332,12.970794893735224) -- (0.7998526635989979,13.311510630777486);
\draw [line width=2pt,dotted,color=ffffff] (0.673550848648332,12.970794893735224) -- (0.9859053446824249,12.785123057595808);
\draw [line width=2pt,dotted,color=ffffff] (-0.20376177332118176,12.660707091929538) -- (-0.07745995837051593,13.001422828971798);
\draw [line width=2pt,dotted,color=ffffff] (-0.20376177332118176,12.660707091929538) -- (0.10859272271291105,12.475035255790122);
\draw (-9.359313157304904,4.560035413955049) node[anchor=north west] {\huge{$\mathbf{Gr_4(X_\lambda^\bullet)}$}};
\draw (-5.69934304237736,9.522706756229361) node[anchor=north west] {\huge{$\mathbf{Gr_2(X_\lambda^\bullet)}$}};
\draw (-7.467294708062698,6.948320997424561) node[anchor=north west] {\huge{$\mathbf{Gr_3(X_\lambda^\bullet)}$}};
\draw (-3.652241113689074,11.476758597249871) node[anchor=north west] {\huge{$\mathbf{Gr_1(X_\lambda^\bullet)}$}};
\draw (-2.3805565822311987,13.554877221827239) node[anchor=north west] {\huge{$\mathbf{Gr_0(X_\lambda^\bullet)}$}};
\begin{scriptsize}
\draw [fill=ffffff] (-1.7768146185999911,12.104708595268635) circle (2.5pt);
\end{scriptsize}
\end{tikzpicture}

}

 \caption{Illustration of the $\varepsilon$-flag $X_\lambda^\bullet$}
    \label{fig:flag}
\end{figure}
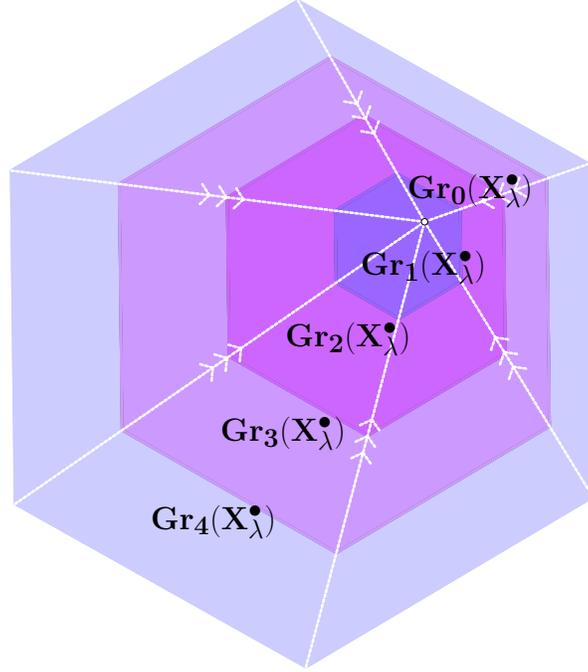

Let us define the following sheaves:

\[F_\varepsilon = \bigoplus_{\lambda \in A} S(X^\bullet_\lambda)^{|C_\lambda|}[(1-\textnormal{sgn}(C_\lambda))/2],\]

\[F = \bigoplus_{\lambda \in A} \kk_{\{x_\lambda\}}^{|C_\lambda|}[(1-\textnormal{sgn}(C_\lambda))/2].\]

Then one has: \begin{align*}
    \chi(F_\varepsilon) &= \chi \left (\bigoplus_{\lambda \in A}S(X^\bullet_\lambda)^{|C_\lambda|}[(1-\textnormal{sgn}(C_\lambda))/2] \right) \\
    &= \sum_{\lambda\in A} C_\lambda \cdot \chi(S(X^\bullet_\lambda))\\
    &= \varphi \quad \textnormal{(Proposition \ref{P:flag}-$(1)$).}
\end{align*}

Similarly:

\[\chi(F) = \sum_{\lambda \in A} C_\lambda \cdot 1_{\{x_\lambda\}}.\]

Moreover, one has by additivity of interleavings (Proposition \ref{p:additivityinterleavings}):

\begin{align*}
    d_C(F_\varepsilon,F) &\leq \max_{\lambda \in A} d_C\left (S(X^\bullet_\lambda)^{|C_\lambda|}[(1-\textnormal{sgn}(C_\lambda))/2] , \kk_{\{x_\lambda\}}^{|C_\lambda|}[(1-\textnormal{sgn}(C_\lambda))/2] \right) \\
    &= \max_{\lambda \in A} d_C\left (S(X^\bullet_\lambda), \kk_{\{x_\lambda\}} \right ) \\ 
    &\leq \frac{\varepsilon}{2} \leq \varepsilon \quad \textnormal{(Proposition \ref{P:flag}-$(2)$).}
\end{align*}

Therefore, one has for all $k>0$: 
\begin{align*}
\delta\left (\varphi, \sum_{\lambda\in A} C_\lambda \cdot 1_{\{x_\lambda\}}\right ) &= \delta\left ( \chi(F), \chi(F_{\frac{1}{k}}) \right). 
\end{align*}

\noindent Since $\delta$ is $d_C$-dominated, $(F_{\frac{1}{k}})_{k> 0}$ is a uniformly bounded sequence of compactly supported constructible sheaves, and $d_C(F, F_{\frac{1}{k}}) \underset{k \longrightarrow +\infty}{\longrightarrow} 0$, we conclude that \[\delta \left (\varphi, \sum_{\lambda\in A} C_\lambda \cdot  1_{\{x_\lambda\}} \right) =0.\]

\end{proof}

\begin{proposition}\label{p:plcase}
Let $\varphi \in \CF_\PL(\V)$ with compact support, and let $x \in \V$. Then one has  
\[\delta \left (\varphi,\left (\int \varphi \dk \right ) \cdot 1_{\{x\}} \right) = 0. \]
\end{proposition}

\begin{proof}
Given $u,v \in \V$, we set $[u,v] = \{t\cdot u + (1-t)\cdot v \mid t \in [0,1]\}$. Let us write $\varphi = \sum_{\lambda \in A} C_\lambda \cdot 1_{X_\lambda}$, with $A$ finite, $C_\lambda \in \Z$ and $X_\lambda$ compact and convex polytopes. For $\lambda \in A$, let $x_\lambda \in X_\lambda$. Then by Lemma \ref{lemma:reunion} applied to $\psi = \sum_{\lambda\in A} C_\lambda \cdot  1_{[x_\lambda,x]}$, one has:

\begin{align*}
    \delta \left (\psi, \sum_{\lambda\in A} C_\lambda \cdot  1_{\{x_\lambda\}}\right ) = 0 =  \delta \left (\psi, \sum_{\lambda\in A} C_\lambda \cdot  1_{\{x\}}\right ).
\end{align*}

Therefore:

\[ \delta \left (\sum_{\lambda\in A} C_\lambda \cdot  1_{\{x_\lambda\}}, \sum_{\lambda\in A} C_\lambda \cdot  1_{\{x\}} \right ) = 0.\]

We now apply Lemma \ref{lemma:reunion} to $\varphi$: 

\begin{align*}
    \delta \left(\varphi, \left (\int \varphi \dk \right ) \cdot  1_{\{x\}}  \right) &= \delta \left(\varphi, \sum_{\lambda\in A} C_\lambda \cdot  1_{\{x\}}  \right)  \\ &\leq \delta \left (\varphi, \sum_{\lambda\in A} C_\lambda \cdot  1_{\{x_\lambda\}} \right)  + \delta \left (\sum_{\lambda\in A} C_\lambda \cdot  1_{\{x_\lambda\}}, \sum_{\lambda\in A} C_\lambda \cdot  1_{\{x\}}  \right) \\
    &= 0.
\end{align*}
\end{proof}

\subsection{General case}\label{section:generalcase}

In this final section, we generalize the previous results to arbitrary stratifications, by piecewise linear approximation (Theorem \ref{th:PLapprox}).

\begin{lemma}\label{lem:gencase}
Let $\varphi \in \CF(\V)$ with compact support, and let $x \in \V$. Then one has 

\[\delta \left ( \varphi,\left ( \int \varphi \dk \right ) \cdot 1_{\{x\}} \right ) = 0. \]
\end{lemma}

\begin{proof}
Let $F \in \D^b_{\R c}(\kk_\V)$ with compact support such that $\varphi = \chi(F)$. According to Theorem \ref{th:PLapprox}, for all $n\in \Z_{> 0}$, there exists $F_n \in \D^b_\PL(\kk_\V)$ such that $d_C(F,F_n) \leq \frac{1}{n}$, $\textnormal{supp}(F_n) \subset T_{\frac{1}{n}}(\textnormal{supp}(F))$, and the sequence $(F_n)$ is uniformly bounded. In particular, $F_n$ has compact support for all $n \geq 1$. Moreover by Proposition \ref{p:sections}, one has for all $n \geq 1$,

\[\Rr \Gamma (\V ; F) \simeq \Rr \Gamma (\V ; F_n). \]

Therefore, $\int \chi(F_n) \dk = \int \varphi \dk$ according to Lemma \ref{lem:integralsection}. Consequently, for all $n >0$: 

\begin{align*}\delta \left (\varphi, \left (\int \varphi \dk \right ) \cdot 1_{\{x\}} \right ) &\leq \delta \left (\varphi, \chi(F_n) \right ) + \delta\left (\chi(F_n), \left (\int \varphi \dk \right ) \cdot 1_{\{x\}}\right) \\
&= \delta \left (\varphi, \chi(F_n) \right ) + \delta\left (\chi(F_n), \left ( \int  \chi(F_n) \dk \right) \cdot 1_{\{x\}}\right) \\ 
&= \delta \left (\varphi, \chi(F_n) \right ) \quad \textnormal{(Proposition \ref{p:plcase})}  \\ 
&= \delta \left (\chi(F), \chi(F_n) \right ).
\end{align*} 

\noindent Since $\delta$ is $d_C$-dominated, $(F_n)$ is a uniformly bounded sequence of constructible compactly supported sheaves, and $d_C(F,F_n) \underset{n \longrightarrow +\infty}{\longrightarrow} 0$, we conclude that:

\[\delta \left (\varphi, \left (\int \varphi \dk \right ) \cdot 1_{\{x\}} \right ) = 0.\]
\end{proof}

\begin{theorem}\label{th:main}
Let $\delta$ be a $d_C$-dominated pseudo-extended metric on $\CF(\V)$, and let $\varphi,\psi \in \CF(\V)$ with compact supports be such that $\int \varphi \dk = \int \psi \dk $. Then: $$\delta(\varphi, \psi) = 0. $$
\end{theorem}

\begin{proof}

By the above lemma, \begin{align*}
    \delta(\varphi,\psi) &\leq \delta\left (\varphi, \left (\int \varphi \dk \right ) \cdot 1_{\{0\}}\right) + \delta\left (\left (\int \psi \dk \right ) \cdot 1_{\{0\}}, \psi\right ) \\ &= 0 \quad \textnormal{(Lemma \ref{lem:gencase})}.
\end{align*}
\end{proof}

\begin{corollary}\label{cor:main}
Let $F,G \in \D^b_{\R c}(\kk_\V)$ with compact support, such that $d_C(F,G) < +\infty$. Then: \[\delta(\chi(F), \chi(G)) = 0.\]

\end{corollary}

\begin{corollary}
Let $X$ be a real analytic manifold, and let $\varphi \in \CF(X)$ with compact support. Also, consider $f,g : X \longrightarrow \V$ some morphisms of real analytic manifolds proper on $\textnormal{supp}(\varphi)$. Then:

\[\delta (f_\ast \varphi, g_\ast \varphi ) = 0.\]
\end{corollary}

\begin{proof}
By \cite[Theorem 2.3]{Sch91}, $f_\ast \varphi$ and $g_\ast \varphi$ are indeed constructible and have compact support by the hypothesis. Let $a_X : X \longrightarrow \{\textnormal{pt}\}$ and $a_\V : \V \longrightarrow \{\textnormal{pt}\}$ be the constant maps. Then by \cite[Section 2]{Sch91}, one has:

\begin{align*}
    \int f_\ast \varphi \dk &= a_{\V \ast}(f_\ast \varphi) \\
    &= (a_{\V}\circ f)_{\ast} \varphi \quad \textnormal{(Theorem \ref{th:functoriality}-3)} \\
    &= a_{X \ast} \varphi \\
    &= \int \varphi \dk.
\end{align*}

Similarly, $\int g_\ast \varphi \dk = \int \varphi \dk = \int f_\ast \varphi \dk$. Since both $f_\ast \varphi$ and $g_\ast \varphi$ have compact support, we conclude the proof by applying Theorem \ref{th:main}.
\end{proof}

\section{Consequences for Topological Data Analysis}

This section is devoted to applying our main result to constructions of Topological Data Analysis (TDA). We start by recalling standard definitions concerning multi-parameter persistence modules, and review results of \cite{BP21} that allows to compare, the categories of persistence modules with $d$ parameters equipped with the interleaving distance $d_I$, and sheaves on a $\R^d$ endowed with the convolution. This bridge allow transferring corollary \ref{cor:main} to the setting of persistence, and to prove that there cannot exist any non-trivial $d_I$-continuous additive invariants of persistence modules. By getting into the persistent world, we are able to apply Lesnick's universality theorem, that allow removing any occurrence of the interleaving distance in the statements. We end the section by applying our results to several common TDA construction.

\subsection{Persistence and Sheaves}

For a general introduction to multi-parameter persistence, we refer the reader to \cite{BL22}. Let $d \geq 0$. We equip $\R^d$ with the partial order $\leq$, defined by $(x_1,...,x_d) \leq (y_1,...,y_d)$ iff $x_i \leq y_i$ for all $i$. We denote $(\R^d, \leq)$ for the associated poset category. The category of \emph{persistence modules} with $d$-parameters, denoted by $\mathrm{Pers}_\kk(\R^d)$, is the category of functors $(\R^d,\leq) \longrightarrow \Mod(\kk)$ and natural transformations.

Persistence modules are usually compared using the interleaving distance, which is defined as follows. Let $\varepsilon\geq 0$ and $M \in \mathrm{Pers}_\kk(\R^d)$. The $\varepsilon$-\emph{shift} of $M$ is the persistence modules $M[\varepsilon]$ defined, for $x \leq y \in \R^d$, by:

\[M[\varepsilon](x) = M(x + (\varepsilon,...,\varepsilon)), ~~~M[\varepsilon](x\leq y ) = M(x + (\varepsilon,...,\varepsilon) \leq y + (\varepsilon,...,\varepsilon)).\]

The collection of linear maps $(M(x \leq x + (\varepsilon,..., \varepsilon)))_{x\in \R^d}$ induces a natural transformation $M \longrightarrow M[\varepsilon]$, denoted $\tau_\varepsilon^M$. An $\varepsilon$-\emph{interleaving} between two persistence modules $M$ and $N$ in $\mathrm{Pers}_\kk(\R^d)$ is the data of two morphisms $f : M \longrightarrow N[\varepsilon]$ and $g : N \longrightarrow M[\varepsilon]$ such that $g[\varepsilon] \circ f = \tau_{2 \varepsilon}^M$ and $f[\varepsilon] \circ g = \tau_{2 \varepsilon}^N$. If there exists an $\varepsilon$-interleaving between $M$ and $N$, we say that they are $\varepsilon$-interleaved, and write $M \sim_\varepsilon N$.

\begin{definition}
The interleaving distance between the persistence modules $M$ and $N$ in $\mathrm{Pers}_\kk(\R^d)$ is the possibly infinite quantity:

\[d_I (M,N) := \inf \{\varepsilon \geq 0 \mid M \sim_\varepsilon N\}.\]
\end{definition}

\begin{remark}
\begin{enumerate}
    \item The interleaving distance is an extended-pseudo metric on the class of objects of $\mathrm{Pers}_\kk(\R^d)$.
    
    \item By exactness of the $\varepsilon$-shift functor, the interleaving distance readily extends to the bounded derived category of persistence modules $\D^b(\mathrm{Pers}_\kk(\R^d))$. In the following, we will still denote it $d_I$.
\end{enumerate}
\end{remark}

We now introduce the $\gamma$-topology after Kashiwara-Schapira \cite[Section 3.5]{KS90}, as an intermediate between the euclidean topology and the downset (or Alexandrov) topology. Let $\gamma = (\R_{\geq 0})^d$. An open set $U \subset \R^d$ is $\gamma$-\emph{open} if it satisfies $U + \gamma = U$. The set of $\gamma$-open subsets of $\R^d$ indeed forms a topology of $\R^d$, named the $\gamma$-topology. We denote the associated topological space by $\R^d_\gamma$. The identity map $\varphi_\gamma :  \R^d \longrightarrow \R^d_\gamma $, $x \mapsto x$, is continuous, and induces an adjunction:

\[\begin{tikzcd}
 \varphi_\gamma^{-1} : \D^b(\kk_{\R^d_\gamma}) \arrow[r,shift left=0.4ex] \arrow[r,draw=none]& 
    \D^b(\kk_{\R^d}) : \Rr \varphi_{\gamma \ast} \arrow[l,shift left=0.75ex].
\end{tikzcd} \]

Following \cite[Section 4.2]{BP21}, it is possible to define an interleaving distance on $\D^b(\kk_{\R^d_\gamma})$, that we write $d_I^\gamma$. Also, we endow $\R^d$ with the norm $\|\cdot\|_\infty$ defined by $\|x\|_\infty := \max_i |x_i|$, and denote by $d_C$ the associated convolution distance on $\D^b(\kk_{\R^d})$.

\begin{theorem}[\cite{BP22} ]\label{th:comparisionconvolution}
For all $F,G \in \D^b(\kk_{\R^d})$, and $H,I \in \D^b(\kk_{\R^d_\gamma})$ one has:
\begin{enumerate}
    \item $ d_I^\gamma(\Rr \varphi_{\gamma \ast} F, \Rr \varphi_{\gamma \ast} G) \leq d_C(F,G)$;
    
    \item $d_C (\varphi_\gamma^{-1} H, \varphi_\gamma^{-1} I ) =d_I^\gamma (H,I)$.
\end{enumerate}

\end{theorem}

Moreover, in \cite{BP21}, the authors introduce a pair of adjoint functors:

\[\begin{tikzcd}
 \alpha^{-1} :   \D^b(\mathrm{Pers}_\kk(\R^d)) \arrow[r,shift left=0.4ex] \arrow[r,draw=none]& 
    \D^b(\kk_{\R^d_\gamma}): \Rr \alpha_\ast \arrow[l,shift left=0.75ex],
\end{tikzcd} \]

and prove the following.

\begin{theorem}[\cite{BP21}]\label{th:comparisons}
For all $F,G \in \D^b(\kk_{\R^d_\gamma})$, and $M,N \in \D^b(\mathrm{Pers}_\kk(\R^d))$, one has:

\begin{enumerate}
    \item $d_I(\Rr \alpha_\ast F, \Rr \alpha_\ast G ) = d_I^\gamma(F,G)$;
    \item $d_I^\gamma( \alpha^{-1}M, \alpha^{-1} N) = d_I(M,N)$.
\end{enumerate}
\end{theorem}

We will also need the following lemma, that was not included in \cite{BP21}.

\begin{lemma}\label{lem:beta}
Let $M\in \D^b(\mathrm{Pers}_\kk(\R^d))$, then $d_I(M, \Rr \alpha_\ast\alpha^{-1} M) = 0$.
\end{lemma}

\begin{proof}
By \cite[Fact 2.10]{BP21} and \cite[Proposition 2.11-(i)]{BP21}, one has \[\alpha^{-1}\circ   \Rr \alpha_\ast \simeq \mathrm{id}_{\D^b(\kk_{\R^d_\gamma})}.\] 

Therefore, for any $M\in \D^b(\mathrm{Pers}_\kk(\R^d))$, by theorem \ref{th:comparisons}-$(1)$, one has:
\begin{align*}
d_I(M, \Rr \alpha_\ast\alpha^{-1} M) &= d_I^\gamma(\alpha^{-1} M, \alpha^{-1} (\Rr \alpha_\ast\alpha^{-1} M)) \\ 
&= d_I^\gamma (\alpha^{-1} M,  (\alpha^{-1} \Rr \alpha_\ast)\alpha^{-1} M)) \\ 
&=  d_I^\gamma (\alpha^{-1} M, \alpha^{-1} M) \\ 
&= 0.
    \end{align*}
\end{proof}

\noindent Combining the above results, we obtain the following adjunction: 

\[\begin{tikzcd}
 (\alpha\circ\varphi_\gamma)^{-1} :  (\D^b(\mathrm{Pers}_\kk(\R^d)), d_I) \arrow[r,shift left=0.4ex] \arrow[r,draw=none]& 
    (\D^b(\kk_{\R^d}), d_C)  : \Rr (\alpha\circ\varphi_\gamma)_\ast \arrow[l,shift left=0.75ex],
\end{tikzcd} \]

\noindent where the left adjoint functor is objectwise distance preserving, and the right adjoint is objectwise $1$-Lipschitz. 

\begin{definition}
A persistence module $M \in \D^b(\mathrm{Pers}_\kk(\R^d))$ is constructible, if $ (\alpha\circ\varphi_\gamma)^{-1} M \in \D^b_{\R c}(\kk_{\R^d})$.
\end{definition}

\begin{remark}
Constructibility of persistence modules is a rather general finiteness condition. Indeed, standard finiteness conditions such as being finitely presented or cubically encoded both imply constructibility (see \cite{Mil20}).
\end{remark}

We denote by $\D^b_{\R c }(\mathrm{Pers}_\kk(\R^d))$ the full subcategory of $\D^b(\mathrm{Pers}_\kk(\R^d))$ whose objects are constructible persistence modules.

\begin{proposition}
The category $\D^b_{\R c }(\mathrm{Pers}_\kk(\R^d))$ is a triangulated sub-category of $\D^b(\mathrm{Pers}_\kk(\R^d))$. 
\end{proposition}

\begin{proof}
This is a direct consequence of $\D^b_{\R c}(\kk_{\R^d})$ being a triangulated category and $ (\alpha\circ\varphi_\gamma)^{-1}$ being a triangulated functor.
\end{proof}

\begin{definition}
A constructible persistence module $M\in \D^b_{ }(\mathrm{Pers}_\kk(\R^d))$ is compactly generated, if there exists $F \in \D^b_{\R c}(\kk_{\R^d})$ compactly supported such that $ M \simeq \Rr (\alpha\circ\varphi_\gamma)_\ast F$.
\end{definition}

\subsection{Non-existence of additive stable invariants of persistence modules}

In this section, we identify $K(\D^b_{\R c}(\kk_{\R^d}))$ with $\CF(\R^d)$, according to the sheaf-function correspondence (Theorem \ref{th:correspondence}). Since $\D^b_{\R c }(\mathrm{Pers}_\kk(\R^d))$ is triangulated, its Grothendieck group is well-defined. We let $\kappa$ be the map $\mathrm{ob}(\D^b_{\R c }(\mathrm{Pers}_\kk(\R^d))) \longrightarrow K(\D^b_{\R c }(\mathrm{Pers}_\kk(\R^d)))$ sending a constructible persistence module to its $K$-class.

\begin{definition}
A pseudo-extended metric $\delta$ on $K(\D^b_{\R c }(\mathrm{Pers}_\kk(\R^d)))$ is said to be $d_I$-dominated if for all uniformly bounded sequences $(M_n) \in \D^b_{\R c }(\mathrm{Pers}_\kk(\R^d)) $ of compactly generated persistence modules, and $M \in \D^b_{\R c }(\mathrm{Pers}_\kk(\R^d))$, one has:

\[d_I(M,M_n) \underset{n \longrightarrow +\infty}{\longrightarrow} 0  ~~\Longrightarrow \delta(\kappa(M),\kappa(M_n)) ~~ \underset{n \longrightarrow +\infty}{\longrightarrow} 0. \]
\end{definition}

\noindent Any triangulated functor $T : \mathscr{C} \longrightarrow \mathscr{C'}$ between triangulated categories, induces a group morphism $K(\mathscr{C}) \longrightarrow K(\mathscr{C'})$, that, for simplicity, we keep denoting by $T$. Given $\delta$ a pseudo-extended metric on $K(\D^b_{\R c }(\mathrm{Pers}_\kk(\R^d)))$, we let $\delta_\ast$ be the pseudo-extended metric defined on $\CF(\R^d)$ by: 

\[ \delta_\ast (\varphi, \psi) := \delta(\Rr (\alpha\circ\varphi_\gamma)_\ast \varphi, \Rr (\alpha\circ\varphi_\gamma)_\ast \psi). \]

\begin{proposition}
The pseudo-extended metric $\delta$ on $K(\D^b_{\R c }(\mathrm{Pers}_\kk(\R^d)))$ is  $d_I$-dominated if and only if $\delta_\ast$ is $d_C$-dominated.
\end{proposition}

\begin{proof}

Assume that $\delta$ is $d_I$ dominated. Let $(F_n)$ and $F$ be compactly supported in $\D^b_{\R c }(\kk_{\R^d})$  constructible sheaves such that, $(F_n)$ is uniformly bounded, and \[d_C(F,F_n) \underset{n \longrightarrow +\infty}{\longrightarrow} 0 .\] 

\noindent Thus, by theorem \ref{th:comparisionconvolution}, \[d_I(\Rr (\alpha\circ\varphi_\gamma)_\ast F,\Rr (\alpha\circ\varphi_\gamma)_\ast F_n) \underset{n \longrightarrow +\infty}{\longrightarrow} 0.\] 

 \noindent The functor $\Rr (\alpha\circ\varphi_\gamma)_\ast$ has finite cohomological dimension \cite[Proposition 3.11]{BP21}, thus, the sequence $(\Rr (\alpha\circ\varphi_\gamma)_\ast F_n)$ is uniformly bounded, and compactly generated by definition.  Since $\delta$ is $d_I$-dominated, we deduce that 
\begin{align*}
    \delta(\kappa(\Rr (\alpha\circ\varphi_\gamma)_\ast F),\kappa(\Rr (\alpha\circ\varphi_\gamma)_\ast F_n))) &= \delta (\Rr (\alpha\circ\varphi_\gamma)_\ast  \chi(F),\Rr (\alpha\circ\varphi_\gamma)_\ast \chi(F_n))) \\
    &= \delta_\ast(\chi(F), \chi(F_n)) \underset{n \longrightarrow +\infty}{\longrightarrow} 0.
\end{align*}

\noindent Therefore, $\delta_\ast$ is $d_C$-dominated. The proof of the converse works similarly.
\end{proof}

\begin{corollary}
Let $\delta$ be a $d_I$-dominated pseudo-extended metric on $K(\D^b_{\R c }(\mathrm{Pers}_\kk(\R^d)))$. Then for all $M,N \in \D^b_{\R c }(\mathrm{Pers}_\kk(\R^d)) $ compactly generated, such that $d_I(M,N) < +\infty$, one has $\delta(\kappa(M), \kappa(N)) = 0$.
\end{corollary}

\begin{proof}
 According to Theorem \ref{th:comparisionconvolution}, one has $d_I(M,N) = d_C ( (\alpha\circ\varphi_\gamma)^{-1} M,  (\alpha\circ\varphi_\gamma)^{-1} N) < +\infty $. Also, since $\delta$ is $d_I$-dominated, $\delta_\ast$ is $d_C$-dominated. Therefore, by corollary \ref{cor:main}, 
 
 \begin{align}
  \delta_\ast(\chi((\alpha\circ\varphi_\gamma)^{-1} M), (\alpha\circ\varphi_\gamma)^{-1} N)) \\ 
  = \delta(\kappa(\Rr(\alpha\circ\varphi_\gamma)_\ast(\alpha\circ\varphi_\gamma)^{-1} M), \kappa(\Rr(\alpha\circ\varphi_\gamma)_\ast(\alpha\circ\varphi_\gamma)^{-1} N)) \\ = 0.
     \end{align}

Note that by \cite[Corollary 1.6]{KS18}, $\Rr \varphi_{\gamma \ast} \circ \varphi_\gamma^{-1} \simeq \mathrm{id}_{\D^b(\kk_{\V_\gamma})}$. Therefore, \[\Rr(\alpha\circ\varphi_\gamma)_\ast(\alpha\circ\varphi_\gamma)^{-1} M \simeq \Rr \alpha_\ast \alpha^{-1} M.\]

\noindent By Lemma \ref{lem:beta}, $d_I(M,\Rr \alpha_\ast \alpha^{-1} M) = 0$. Since $\delta$ is $d_I$-dominated, $\delta(\kappa(M), \kappa(\Rr \alpha_\ast \alpha^{-1} M)) = 0$. Similarly, $\delta(\kappa(N), \kappa(\Rr \alpha_\ast \alpha^{-1} M)) = 0$. From equation $(4.3)$, we deduce that 

\begin{align*}
  &   \delta(\kappa(M), \kappa(N))  \\
    & \leq \delta(\kappa(M), \kappa(\Rr \alpha_\ast \alpha^{-1} M)) + \delta(\kappa(\Rr \alpha_\ast \alpha^{-1} M), \kappa(\Rr \alpha_\ast \alpha^{-1} N)) \\ & + \delta(\kappa(\Rr \alpha_\ast \alpha^{-1} N), \kappa(N)) \\
    & = 0 + 0 + 0.
\end{align*}

\end{proof}

Let $(G,+)$ be an abelian group, endowed with a pseudo-extended metric $\delta$. We think of $G$ as a group of invariants of persistence modules, and of $\delta$ as a way of measuring dissimilarity between invariants.

\begin{definition}
Let $\mathscr{A}$ be an abelian category. A map  $\lambda : \mathrm{ob}(\D^b(\mathscr{A})) \longrightarrow G$ is said to be additive with respect to exact triangles if for all exact triangles $X \longrightarrow Y \longrightarrow Z \stackrel{+1}{\longrightarrow}$ in $\D^b(\mathscr{A})$, one has, $\lambda(Y) = \lambda(X) + \lambda(Z)$. 

\noindent Similarly, a map $\lambda : \mathrm{ob}(\mathscr{A}) \longrightarrow G$ is said to be additive with respect to short exact sequences if for all short exact sequences $0 \longrightarrow X \longrightarrow Y \longrightarrow Z \longrightarrow 0$ in $\mathscr{A}$, one has, $\lambda(Y) = \lambda(X) + \lambda(Z)$. 
\end{definition}

\begin{theorem}\label{th:derivedpersistence}
     Let $\lambda : \mathrm{ob}(\D^b_{\R c }(\mathrm{Pers}_\kk(\R^d))) \longrightarrow G$ be an additive map with respect to exact triangles,  such that for all uniformly bounded sequence $(M_n)$, and $M$,  compactly generated constructible persistence modules in $\D^b_{\R c }(\mathrm{Pers}_\kk(\R^d))$, \[d_I(M,M_n)\underset{n \longrightarrow +\infty}{\longrightarrow} 0 \Longrightarrow \delta(\lambda(M), \lambda(M_n)) \underset{n \longrightarrow +\infty}{\longrightarrow} 0.\] Then for all compactly generated constructible persistence modules $M,N \in \D^b_{\R c }(\mathrm{Pers}_\kk(\R^d))$:  \[d_I(M,N) < +\infty \Longrightarrow \delta (\lambda(M), \lambda(N)) = 0.\]

\end{theorem}

\begin{proof}
 By the universal property of the Grothendieck group, there exists a unique group morphism $\Phi : K(\D^b_{\R c}(\mathrm{Pers}_\kk(\R^d))) \longrightarrow G$ such that the following diagram of maps:
 
 \[\xymatrix{\mathrm{ob}(\D^b_{\R c}(\mathrm{Pers}_\kk(\R^d))) \ar[rr]^\lambda \ar[rd]_\kappa & & G \\
 & K(\D^b_{\R c}(\mathrm{Pers}_\kk(\R^d))) \ar[ru]_\Phi &} \]
 
 \noindent is commutative. For $\varphi,\psi \in K(\D^b_{\R c}(\mathrm{Pers}_\kk(\R^d))) $, define  $\delta^{-1}(\varphi, \psi) := \delta(\Phi(\varphi), \Phi(\psi)).$ It is a pseudo-extended metric on $K(\D^b_{\R c}(\mathrm{Pers}_\kk(\R^d)))$. Moreover, $\delta^{-1}$ is $d_I$-dominated by commutativity of the above diagram. Therefore, by theorem \ref{th:derivedpersistence}, for all $M,N \in \D^b_{\R c}(\mathrm{Pers}_\kk(\R^d))$ compactly generated, such that $d_I(M,N) < +\infty$, one has: \[\delta^{-1}(\kappa(M), \kappa(N)) = \delta(\Phi(\kappa(M)),\Phi(\kappa(N)) = \delta(\lambda(M), \lambda(N)) = 0.\]
\end{proof}

\begin{theorem} \label{th:mainpers}
     Let $\lambda : \mathrm{ob}(\mathrm{Pers}_{\kk, \R c}(\R^d)) \longrightarrow G$ be an additive map with respect to short exact sequences,  such that for all $(N_n)$ uniformly bounded sequence, and $N$,  compactly generated constructible persistence modules in $\mathrm{Pers}_{\kk, \R c}(\R^d))$, 
     \[d_I(N,N_n)\underset{n \longrightarrow +\infty}{\longrightarrow} 0\Longrightarrow \delta(\lambda(N), \lambda(N_n)) \underset{n \longrightarrow +\infty}{\longrightarrow} 0. \] 
     Also, assume that the sum map $G \times G \longrightarrow G$ and the symmetric map $G \longrightarrow G$ are continuous with respect to $\delta$. Then, for all compactly generated constructible persistence modules $M,N \in\mathrm{Pers}_{\kk, \R c}(\R^d)$, if $d_I(M,N) < +\infty$, then $\delta (\lambda(M), \lambda(N)) = 0$.
\end{theorem}

\begin{proof}
Let $\lambda$ be as in the statement of the theorem. We extend it as a map $\overline{\lambda} : \mathrm{ob}(\D^b_{\R c}(\mathrm{Pers}_\kk(\R^d))) \longrightarrow G$, by $\overline{\lambda}(M) := \sum_i (-1)^i \lambda(\Ho^i(M))$ (the sum is always finite by boundedness assumption, hence well-defined in $G$). One checks easily that $\overline{\lambda}$ is additive with respect to exact triangles. Moreover, since the $\varepsilon$-shift functor is exact, the cohomological functors $\Ho^i$ preserve interleavings. Therefore, for  $(N_n)$ and $N$ compactly generated persistence modules in $\D^b_{\R c}(\mathrm{Pers}_\kk(\R^d))$, if $d_I(N,N_n)\underset{n \longrightarrow +\infty}{\longrightarrow} 0$, then for all $i\in \Z$, $d_I(\Ho^i(N),\Ho^i(N_n))\underset{n \longrightarrow +\infty}{\longrightarrow} 0$. Therefore, by hypothesis, \[\delta(\lambda(\Ho^i(N)),\lambda(\Ho^i(N_n))) \underset{n \longrightarrow +\infty}{\longrightarrow} 0.\] Assume in addition that $(N_n)$ is uniformly bounded, that is, there exists $C \in \Z_{\geq 0}$ such that for all $n \geq 0$, $\Ho^i(N_n) \simeq 0$ whenever $|i |> C$. We also assume without loss of generality that $\Ho^i(N) \simeq 0$ for all $|i |> C$. Therefore, $\overline{\lambda}(N) = \sum_{i = -C}^C (-1)^i \lambda(\Ho^i(N))$, and for all $n \geq 0$, $\overline{\lambda}(N_n) =  \sum_{i = -C}^C (-1)^i \lambda(\Ho^i(N_n))$. By continuity of the group operations with respect to $\delta$, one has:

\[\delta(\overline{\lambda}(N),\overline{\lambda}(N_n)) \underset{n \longrightarrow +\infty}{\longrightarrow} 0.\]

Therefore, $\overline{\lambda}$ satisfies the hypothesis of theorem \ref{th:derivedpersistence}, and we can conclude  for all compactly generated persistence modules $M,N \in \mathrm{Pers}_{\kk, \R c}(\R^d)$, if $d_I(M,N)< +\infty$, then:

\[\delta(\lambda(M),\lambda(N)) = \delta (\overline{\lambda}(M), \overline{\lambda}(N)) = 0.\]

\end{proof}

Our effort to formulate the consequences of our main result in a purely persistent and non-derived setting, allows using the universality result proved by Lesnick in \cite[Corollary 5.6]{Les15}, that we recall here. Given $f : X \longrightarrow \R^d$, its sub-levelsets filtration is the functor $\mathcal{S}(f) : (\R^d, \leq ) \longrightarrow \textbf{Top}$ defined by $\mathcal{S}(f)(x) := f^{-1} \{s \in \R^d \mid s \leq x\}$, and its $i$-th persistence module is the functor $\mathcal{S}_i(f) :=\Ho_i(- ; \kk) \circ \mathcal{S}(f)$, with  $\Ho_i(- ; \kk)$ the $i$-th singular homology with coefficients in $\kk$ functor. 

For $f : X \longrightarrow \R^d$ and $g : X \longrightarrow \R^d$ two continuous maps of topological space, one sets:

\[d_\infty (f,g) := \inf_{h : X \stackrel{\sim}{\longrightarrow} Y} \sup_{x\in X} \|f(x) - g \circ h (x)\|_\infty,\] where $h$ ranges over all homeomorphisms from $X$ to $Y$.

\begin{theorem}[\cite{Les15}]\label{th:universality}
     Let $\kk$ be a prime field, and $d$ be a pseudo-extended metric on $\mathrm{ob}(\mathrm{Pers}_\kk(\R^d))$, such that for all maps of topological spaces $f : X \longrightarrow \R^d$ and $g : Y \longrightarrow \R^d$, one has \[d(\mathcal{S}_i(f), \mathcal{S}_i(g)) \leq d_\infty(f,g).\]
     
     Then $d \leq d_I$.
\end{theorem}

Note that Lesnick's proof relies on the existence of geometric lift for interleavings of persistence modules \cite[Proposition 5.8]{Les15}, which holds without any assumption on the persistence modules. Therefore, theorem \ref{th:universality} restricts to constructible persistence modules in the following way.

\begin{theorem}[Universality, constructible version] \label{th:universalityconstructible} Let $\kk$ be a prime field, and $d$ be a pseudo-extended metric on $\mathrm{ob}(\mathrm{Pers}_{\kk, \R c}(\R^d))$, such that for all maps of topological spaces $f : X \longrightarrow \R^d$ and $g : Y \longrightarrow \R^d$, such that $\mathcal{S}_i(f)$ and $\mathcal{S}_i(g))$ are constructible, one has \[d(\mathcal{S}_i(f), \mathcal{S}_i(g)) \leq d_\infty(f,g).\]
     
     Then $d \leq d_I$.

\end{theorem}

Combining Lesnick's universality theorem with our theorem \ref{th:mainpers}, we obtain the following corollary. 

\begin{corollary} 
 Let $\lambda : \mathrm{ob}(\mathrm{Pers}_{\kk, \R c}(\R^d)) \longrightarrow G$ be an additive map with respect to short exact sequences, such that for all maps of topological spaces $f : X \longrightarrow \R^d$ and $g : Y \longrightarrow \R^d$, one has \[\delta(\lambda(\mathcal{S}_i(f)), \lambda(\mathcal{S}_i(g))) \leq d_\infty(f,g).\]
 
 Then for all compactly generated persistence modules $M,N \in \mathrm{Pers}_{\kk, \R c}(\R^d)$, one has 
 
 \[d_I(M,N) < +\infty \Longrightarrow \delta(\lambda(M), \lambda(N)) = 0.\]
\end{corollary}

\subsection{Examples}

In this section, we apply our results to well-known constructions that are common to TDA. We still denote by $\V$ a finite dimensional real vector space endowed with a norm $\|\cdot\|$.

\subsubsection{Radon Transforms}

Radon Transforms are a general class of transformations on constructible functions, for which there exists a well-formulated criterion of invertibility \cite{Sch95}. The Euler Characteristic Transform (ECT) is a particular instance of invertible Radon transform, which has found numerous applications \cite{BG09, TMS14, CMT18, Crawford_2019}.

Let $X$ and $Y$ be two real analytic manifolds, and let $S \subset X \times Y$ be a locally closed subanalytic subset. Let $q_1$ and $q_2$ be the first and second projection defined on $X \times Y$. We shall assume the following hypothesis:

\begin{equation}\label{hyp:properness}
    q_2 \textnormal{~is proper on the closure of~} S \mathrm{~in}~ X \times Y.
\end{equation}

\begin{definition}
The Radon transform associated to $S$, is the group homomorphism $\mathscr{R}_S : \CF(X) \longrightarrow \CF(Y)$ defined by:

\[ \mathscr{R}_S(\varphi) := q_{2 \ast} \left [ (\varphi \circ q_1) \cdot 1_S \right ].  \]
\end{definition}

\begin{remark}
When $X = \V$, $Y = \mathbb{S}^\ast \times \R$, and $S_{ECT} = \{(v, (\xi, t)) \in X \times Y \mid \xi(v) \leq t\}$, the transform $\mathscr{R}_{S_{ECT}}$ is the usual Euler Characteristic Transform. We have denoted $\mathbb{S}^\ast$ the unit dual sphere of $\V^\ast$.
\end{remark}

\begin{proposition}
Let $X = \V$, and $Y$ be real analytic manifold, and let $S \subset X \times Y$ satisfying hypothesis (\ref{hyp:properness}). Let $\delta$ be a pseudo-extended distance on $\CF(Y)$, such that $\delta \circ (\mathscr{R}_S \times \mathscr{R}_S)$ is $d_C$-dominated. Then for all $\varphi,\psi \in \CF(\V)$ with compact support such that $\int \varphi \mathrm{d}\chi = \int \psi \mathrm{d}\chi$, one has: \[\delta\left(\mathscr{R}_S(\varphi), \mathscr{R}_S(\psi) \right) = 0.\]
\end{proposition}

\begin{proof}
This is a straightforward consequence of theorem \ref{th:main}.
\end{proof}


    





     



\subsubsection{Amplitudes}

In \cite{Amplitudes}, the authors introduce the notion of amplitude on an abelian category $\mathscr{A}$, as a notion of measurement of the size of objects of $\mathscr{A}$, compatible with exact sequences.

\begin{definition}
Let $\mathscr{A}$ be an abelian category. An amplitude on $\mathscr{A}$ is a class function $\lambda : \mathrm{ob}(\mathscr{A}) \longrightarrow [0, + \infty]$ satisfying $\lambda(0) = 0$, and for all short exact sequence $0 \longrightarrow A \longrightarrow B \longrightarrow C \longrightarrow 0$:

\begin{center}
    \begin{enumerate}
    \item $\lambda(A) \leq \lambda(B)$;
    \item $\lambda(C) \leq \lambda(B)$;
    \item $\lambda(B) \leq \lambda(A) + \lambda(C)$.
\end{enumerate}
\end{center}

\noindent The amplitude $\lambda$ will be said to be additive if $(3)$ is an equality.
\end{definition}

\begin{proposition}\label{P:negamplitude}
Let $\lambda$ be an additive amplitude on $\mathrm{Pers}_{\kk, \R c}(\R^d)$, such that for all  $(M_n)$ and $M$ compactly generated in $\mathrm{Pers}_{\kk, \R c}(\R^d)$, one has:

\[d_I(M_n,M) \underset{n \longrightarrow +\infty}{\longrightarrow} 0 \Longrightarrow |\lambda(M_n) - \lambda(M)| \underset{n \longrightarrow +\infty}{\longrightarrow} 0. \]

\noindent Then for all compactly generated persistence module $M, N \in \mathrm{Pers}_{\kk, \R c}(\R^d)$ such that $d_I(M,N) < +\infty$, one has $\lambda(M) = \lambda(N)$.
\end{proposition}

\begin{proof}
We apply theorem \ref{th:mainpers} to the additive amplitude $\lambda$, where $G = (\R,+)$ is endowed with the standard metric. Thus, for all $M,N \in \mathrm{Pers}_{\kk, \R c}(\R^d)$ compactly generated, $|\lambda(M) - \lambda(N)| = 0$.    
\end{proof}

\begin{corollary}\label{C:amplitudes}
Assume that $\kk$ is a prime field, and let $\lambda : \mathrm{ob}(\mathrm{Pers}_{\kk, \R c}(\R^d)) \longrightarrow [0,+\infty]$ be an additive amplitude on persistence modules such that for all maps of topological spaces $f : X \longrightarrow \R^d$ and $g : Y \longrightarrow \R^d$, and all $i \in \Z_{\geq 0}$, if $\mathcal{S}_i(f)$ and $\mathcal{S}_i(g)$ are constructible, then:
 
 \[ |\lambda(\mathcal{S}_i(f)) - \lambda(\mathcal{S}_i(g))| \leq d_\infty(f,g). \]
 
 \noindent Then for all compactly generated and constructible $M,N \in  \mathrm{Pers}_{\kk, \R c}(\R^d)$, one has \[\lambda(M) = \lambda(N).\]
\end{corollary}

\subsubsection{Additive vectorizations}

It is well-known that persistence modules endowed with the interleaving distance do not embed isometrically into any Hilbert space. Nevertheless, since most machine learning techniques take as input elements of a vector space, it is a very common strategy to define a so-called vectorization of persistence modules, that is, a map $\Phi : \mathrm{ob}(\mathrm{Pers}_{\kk, \R c}(\R^d)) \longrightarrow \mathbb{W}$, where $\mathbb{W}$ is a real vector space, usually endowed with a norm $\|\cdot\|$. We will say that $\Phi$ is additive if it is additive with respect to short exact sequence of persistence modules.

\begin{proposition}\label{P:vectorization}
Let $\Phi : \mathrm{ob}(\mathrm{Pers}_{\kk}(\R^d)) \longrightarrow (\mathbb{W}, \| \cdot \|)$ be an additive vectorization of persistence modules satisfying for all  $(M_n)$ and $M$ in $\mathrm{Pers}_{\kk, \R c}(\R^d)$:

\[d_I(M_n,M) \underset{n \longrightarrow +\infty}{\longrightarrow} 0 \Longrightarrow \|\Phi(M_n) - \Phi(M)\| \underset{n \longrightarrow +\infty}{\longrightarrow} 0. \]

Then for all compactly generated and constructible persistence modules $M,N \in \mathrm{Pers}_{\kk}(\R^d) $, if $d_I(M,N) < +\infty$, then $\Phi(M) = \Phi(N)$.
\end{proposition}

\begin{proof}
The proof is similar to Proposition \ref{P:negamplitude}.
\end{proof}

\begin{corollary}\label{C:vectorization}
Assume that $\kk$ is a prime field, and let $\Phi : \mathrm{ob}(\mathrm{Pers}_{\kk, \R c}(\R^d)) \longrightarrow (\mathbb{W}, \| \cdot \|)$ be an additive vectorization of persistence modules such that for all maps of topological spaces $f : X \longrightarrow \R^d$ and $g : Y \longrightarrow \R^d$, and all $i \in \Z_{\geq 0}$, one has:
 
 \[ \|\Phi(\mathcal{S}_i(f)) - \Phi(\mathcal{S}_i(g))\| \leq d_\infty(f,g). \]
 
Then for all compactly generated and constructible persistence modules $M,N \in \mathrm{Pers}_{\kk}(\R^d) $, if $d_I(M,N) < +\infty$, then $\Phi(M) = \Phi(N)$.
\end{corollary}

\begin{remark}
The construction $\|\Phi(\cdot)\|$, where $\Phi$ is an additive vectorization of persistence modules, provides a very general mean of defining not necessarily additive amplitudes of persistence modules. One interpretation of proposition \ref{P:vectorization} and corollary \ref{C:vectorization} is that such amplitudes can never be reasonably controlled, either by the interleaving distance $d_I$ on persistence modules, nor by the infinite distance $d_\infty$ on functions.
\end{remark}

\section{Discussion and further work}

Our main results Theorem \ref{th:main} and Corollary \ref{cor:main} show that any distance on the group of constructible that can be controlled by the convolution distance --in the sense of domination-- vanishes as soon as two compactly supported constructible functions $\varphi,\psi$ have the same Euler integral, a condition that is satisfied whenever there exists two sheaves $F,G \in \D^b_{\R c}(\kk_\V)$ satisfying $d_C(F,G) < +\infty $ and such that  $\varphi = \chi(F)$ and $\psi = \chi(G)$. The convolution distance is usually interpreted as a $\ell_\infty$ type metric, because of the form of stability (Theorem \ref{th:stability}) it satisfies. Our results therefore give a strong negative incentive on the possibility of a obtaining a $\ell_\infty$-control of the pushforward operation on constructible functions.

In terms of Topological Data Analysis, our result shows that additive invariants of persistence modules cannot be stable with respect to the interleaving distance. Therefore, in order to obtain well-behaved invariants, it is necessary to either loosen the additivity assumption, or to consider stability with respect to other type of distances. These are both active fields of research \cite{bubenik2018algebraic,skraba2020wasserstein,BOOS22}, for which we hope that the present article will highlight the importance.

Schapira recently introduced the concept of constructible functions up to infinity \cite{schapira2021constructible}, that allows one to define Euler integration of constructible functions without compact support. We conjecture that Theorem \ref{th:main} holds when replacing \emph{with compact support} by \emph{constructible up to infinity}, though  we do not know how to adapt our  $\varepsilon$-flag technique to this setting.

\bibliographystyle{amsplain}


\bibliography{biblio}

\end{document}